\documentclass[a4paper]{article}

\usepackage{verbatim}
\usepackage{hyperref}
\usepackage{amsmath}
\usepackage{enumerate, amsthm}
\usepackage{mathrsfs}
\usepackage{xcolor}
\usepackage[margin=1 in]{geometry}
\usepackage{mathtools}
\usepackage{tikz}
\usetikzlibrary{patterns}
\usetikzlibrary{arrows}
\usepackage{ifthen}
\usepackage{bm}
\usepackage{amssymb}
\usepackage{centernot}

\usepackage{calrsfs}

\usepackage{caption}
\usepackage{subcaption}

\numberwithin{equation}{section}

\theoremstyle{plain}
\newtheorem{theorem}{Theorem}[section]
\newtheorem{lemma}[theorem]{Lemma}
\newtheorem{corollary}[theorem]{Corollary}
\newtheorem{proposition}[theorem]{Proposition}

\theoremstyle{definition}

\theoremstyle{remark}
\newtheorem{remark}[theorem]{Remark}

\renewcommand{\P}{\mathbb P}
\newcommand{\E}{\mathbb E}
\newcommand{\cE}{\mathcal E}
\newcommand{\R}{\mathbb R}
\renewcommand{\Re}{\mathcal R}
\newcommand{\Z}{\mathbb Z}
\newcommand{\N}{\mathbb N}
\newcommand{\ep}{\varepsilon}

\newcommand{\cG}{\pazocal G}
\newcommand{\cA}{\mathcal A}
\newcommand{\fp}{\mathfrak p}
\newcommand{\Vol}{\mathrm{Vol}}

\newcommand{\Lu}{\underline{\Lambda}}
\newcommand{\Lc}{\Lambda}

\newcommand{\si}[1]{{\mathchoice{}{}{\scriptscriptstyle}{}#1}}
\newcommand{\sN}{{\mathchoice{}{}{\scriptscriptstyle}{}N}}

\newcommand{\lr}[4]{#3\xleftrightarrow[#1]{#2} #4}
\newcommand{\nlr}[4]{#3\centernot{\xleftrightarrow[#1]{#2}} #4}

\DeclareMathAlphabet{\pazocal}{OMS}{zplm}{m}{n}

\title{Supercritical sharpness for Voronoi percolation}
\author{Barbara Dembin\footnotemark[1]\footnote{IRMA, barbara.dembin@math.unistra.fr} \and Franco Severo\footnotemark[2]\footnote{ICJ, severo@math.univ-lyon1.fr}}

\date{}

\begin{document}
\thispagestyle{empty}
\maketitle

\begin{abstract}
We prove that the supercritical phase of Voronoi percolation on $\R^d$, $d\geq 3$, is well behaved in the sense that for every $p>p_c(d)$ local uniqueness of macroscopic clusters happens with high probability. As a consequence, truncated connection probabilities decay exponentially fast and percolation happens on sufficiently thick 2D slabs. This is the analogue of the celebrated result of Grimmett \& Marstrand for Bernoulli percolation and serves as the starting point for renormalization techniques used to study several fine properties of the supercritical phase. 
\end{abstract}

\section{Introduction}\label{sec:intro}

\textbf{Motivation.} Sharpness of phase transition is a fundamental property in the study of off-critical behavior in percolation theory. In general terms, this property states that the model is ``well behaved'' on the \emph{entire} subcritical and supercritical phases in the sense that finite clusters are in fact ``tiny'' (typically, exponentially decaying size distribution). For Bernoulli percolation on $\Z^d$, $d\geq3$, \emph{subcritical sharpness} was proved independently by Aizenman \& Barsky \cite{AizBar87} and Menshikov \cite{Men86}, while \emph{supercritical sharpness} was obtained by Grimmett \& Marstrand \cite{GriMar90} -- the case $d=2$ was previously known as a consequence of \cite{Kes80}. Thanks to these highly influential works, both subcritical and supercritical phases of Bernoulli percolation on $\Z^d$ are now rather well understood -- see the end of this section for a discussion on the multiple applications of sharpness.

The study of sharpness of phase transition for correlated percolation models in dimensions $d\geq3$ has seen great progress in the last years.  On the one hand, subcritical sharpness has been proved for various models including Voronoi percolation \cite{DumRaoTas17a}, random cluster (or FK) model with $q\geq1$ \cite{DumRaoTas17b} (the special case of FK-Ising model, i.e.~$q=2$, had been obtained in \cite{AizBarFer87}), occupied and vacant sets of Boolean percolation (under some assumptions on the radius distribution) \cite{DumRaoTas17c,DT22}, level-set percolation of some Gaussian fields \cite{DGRS19,Sev21,Mui22} and the vacant set of random interlacements \cite{DGRST23a}. 
Supercritical sharpness on the other hand is known to hold for a considerably smaller subset of the models mentioned above, namely the FK-Ising model \cite{Bod05}, the occupied set of Boolean percolation \cite{Tan93,DT22}, level-set percolation of some Gaussian fields \cite{DGRS19,Sev21} and the vacant set of random interlacements \cite{DGRST23a}.
This discrepancy between our understanding of the subcricital and the supercritical phases of correlated models may be explained by the existence of techniques coming from the sharp threshold theory of Boolean functions, such as Talagrand and OSSS inequalities \cite{Tal94,OSSS}, that have been shown to be very powerful to prove subcritical sharpness, whereas the number of techniques to prove supercritical sharpness remains rather limited -- see however \cite{CMT21} for a new proof of supercritical sharpness for Bernoulli percolation using sharp threshold inequalities. \\ 

\textbf{Main results.} In the present article, we study the supercritical phase of Voronoi percolation. First introduced in the context of first passage percolation \cite{MR1094141}, Voronoi percolation is one of the most studied models in percolation theory. The model can be defined as follows: 
consider the Voronoi tessellation associated to a Poisson point process of intensity $1$ on $\R^d$ and then color independently each cell black (open) with probability $p$ and white (closed) with probability $1-p$ -- see Sections~\ref{sec:preliminaries} and \ref{sec:comparison} for precise definitions. We are interested in the connectivity properties of the random set $\pazocal{V}(p)\subset\R^d$ given by the union of the black cells as $p$ varies.
We say that $\pazocal{V}(p)$ \emph{percolates} if it contains an infinite (i.e.~unbounded) cluster (i.e.~connected component). The model exhibits a (non-trivial) phase transition at the \emph{critical point} defined as
\begin{equation}\label{eq:def_pc}
p_c=p_c(d)\coloneqq\inf\big\{p\in \R:~ \P[\pazocal{V}(p) \text{ percolates}]>0\big\}.
\end{equation}

As it is often the case in percolation theory, the model is rather well understood in dimension $d=2$ due to duality. In fact, it has been proved in \cite{BolRio06} that $p_c(2)=1/2$ and that both subcritical and supercritical sharpness hold -- this result can be seen as the analogue of \cite{Kes80}. Furthermore, a lot is  known about the behaviour at the critical point $p=1/2$, see for instance \cite{Tas14b,AGMT16,MR3878133,Van19,BenSch98}.

In higher dimensions though, much less is known. The existence of a non-trivial phase transition (i.e.~$p_c\in (0,1)$) has been proved in \cite{MR2127372} along with asymptotic bounds for $p_c(d)$. As mentioned above, subcritical sharpness has been recently proved \cite{DumRaoTas17a}. In this paper, we give the first proof of the corresponding result on the supercritical phase. 

\begin{theorem}[Supercritical sharpness] \label{thm:sharpness} 
For every $d\geq3$ and $p>p_c(d)$, there exists $c=c(d,p)>0$ such that the following inequalities hold for every $R\geq 1$, 
\begin{align}\label{eq:exp_decay_diam}
    \P[R\leq \mathrm{Diam}(\mathcal{C}_0(p))<\infty]
    &\leq e^{-cR},\\
    \label{eq:exp_decay_vol}
    \P[R\leq \mathrm{Vol}(\mathcal{C}_0(p))<\infty]
    &\leq e^{-cR^{\frac{d-1}{d}}},
\end{align}
where $\mathcal{C}_0(p)$ denotes the cluster of $0$ in $\pazocal{V}(p)$ and $\mathrm{Diam}(\mathcal{C}_0(p))$ and $\mathrm{Vol}(\mathcal{C}_0(p))$ stand for its diameter and volume, respectively. 
Furthermore, there exists $M=M(d,p)>0$ such that $\pazocal{V}(p)$ percolates in the slab $\R^2\times[0,M]^{d-2}$.
\end{theorem}

The main difficulty in proving Theorem~\ref{thm:sharpness} is to establish a finite size criterion throughout the supercritical phase. This is provided by the following result.
Below, we denote the $\ell^\infty$ box of radius $r$ centered at $0$ by $\Lc_r\coloneqq[-r,r]^d$.

\begin{theorem}[Local uniqueness]\label{thm:unique}
For every $d\geq 3$ and $p>p_c(d)$, one has
    \begin{equation}\label{eq:unique}
    \P[\pazocal{U}(L,p)] \rightarrow 1 ~\text{ as } L\to\infty,
    \end{equation}
where $\pazocal{U}(L,p)$ denotes the event that $\pazocal{V}(p)\cap \Lc_{2L}$ contains only one cluster crossing the annulus $\Lc_L\setminus \Lc_{L/2}$.
\end{theorem}
\begin{remark}
    Theorem~\ref{thm:unique} can be seen as our main theorem, from which Theorem~\ref{thm:sharpness} follows via standard renormalization arguments. Furthermore, with similar argument, it is not hard to deduce from \eqref{eq:unique} that the rate of convergence is exponential, namely $\P[\pazocal{U}(L,p)]\geq 1-e^{-cL}$ for every $L\geq1$ and $p>p_c(d)$, where $c=c(d,p)>0$ -- see Remark~\ref{rk:section 4}.\\
\end{remark}

\textbf{Applications.} Our results open the way to a deeper understanding of the supercritical behavior of Voronoi percolation. Indeed, Theorem~\ref{thm:unique} says that a sort of \emph{finite size criterion} holds throughout the supercritical phase, which is the starting point for renormalization techniques widely used in the study of fine properties of the supercritical phase, thus having multiple applications. For Bernoulli percolation, local uniqueness has been a crucial ingredient in the study of metric properties of the (unique) infinite cluster $\mathcal{C}_\infty$ \cite{AntPis96,GarMar04,GarMar07,GMPT,Dem21,CerfDembin}, the behavior of the random walk on $\mathcal{C}_\infty$ \cite{GriKesZha93, Bar04, SidSzn04, MatPia07, BerBis07, ArmDar16}, analyticity of the percolation density $\theta(p)$ and other observables \cite{GeoPan20}, and large deviation questions \cite{Pis96}, such as in the proof of convergence of large finite clusters to the so-called Wulff crystal \cite{Cer00}. 
Analogous results for Voronoi percolation are now within reach thanks to Theorem~\ref{thm:unique}. \\

\textbf{Strategy of proof.} The proof of Theorem~\ref{thm:unique} follows an \emph{interpolation strategy}, which has been recently used to prove (both subcritical and supercritical) sharpness for level set percolation of Gaussian fields \cite{DGRS19, Sev21} and the vacant set of random interlacements \cite{DGRST23a,DGRST23b,DGRST23c}. 
In a high level, this approach is suitable to models that can be very well approximated locally by a family of finite-range models. However, in implementing this strategy, one runs into several issues coming from the local correlations of the model. Both tasks of finding an effective finite range approximation and controlling the local correlations are highly model-dependent and can vary in difficulty.

We now describe in more details each step of this strategy. First, one constructs a family of \emph{truncated} models $(\pazocal{V}_{N})_{N}$ -- i.e.~$\pazocal{V}_N$ for which regions within distance larger than $N$ are independent -- that well approximates the original model $\pazocal{V}$ locally and for which one can prove sharpness by adapting classical techniques. Then one proves a \emph{global comparison} result -- which may be of independent interest, see Theorem~\ref{thm:comparison} -- stating that connection probabilities for $\pazocal{V}$ and $\pazocal{V}_N$ are comparable at \emph{arbitrary scales}, provided the use a small (for $N$ large) sprinkling, i.e.~a slight change in the parameter $p$. 
The proof of this global comparison result, which is the heart of the approach, goes by looking at the function $(N,p)\mapsto \P[\pazocal{V}_N(p)\in A]$, for $A$ a connection event, and showing that some sort of ``derivative'' in $N$ is much smaller than the derivative in $p$. This inequality is proved by comparing different types of pivotality events through a delicate \emph{local surgery procedure}, where the pivotality event corresponding to the derivative in $N$ comes with an expensive ``local disagreement'' price stemming from the fact that $\pazocal{V}_N$ locally approximates $\pazocal{V}$ well enough. This technique is somewhat reminiscent of the differential inequalities used by Aizenman \& Grimmett \cite{AizGri91} to prove strict inequalities for critical points of enhanced percolation models -- see also \cite{menshikovstrict}. 

A direct consequence of the global comparison result is that one can transfer to $\pazocal{V}$ all the bounds on connection probabilities which are valid for $\pazocal{V}_N$ due the sharpness of its phase transition. However, local uniqueness is not a connection event (in fact, it is not even monotonic). In order to prove local uniqueness, we first observe that disconnection probabilities decay very fast in the whole supercritical regime $p>p_c$ -- in our case, the decay is exponential in the surface, see Corollary~\ref{cor:decay_disconnection}. This implies that, with high probability, every box of radius $\ell=C(\log L)^{\frac{1}{d-1}}$ in $\Lc_L$ is connected to infinity in $\pazocal{V}(p)$, leading to a sort of coarse version of the ``everywhere percolating'' picture from \cite{benjaminitassion17}, where it is proved that, with a small Bernoulli sprinkling, all these giant components get connected to each other. We adapt the proof of \cite{benjaminitassion17} to our context in order to obtain the desired local uniqueness event -- other adaptations of their techniques can be found in \cite{DGRS19,CMT21,DGRST23b} as well. Once local uniqueness is proved, we deduce Theorem~\ref{thm:sharpness} through a renormalization argument, as mentioned above.

In implementing the strategy above, one runs into problems that arise due to (rare, but existent!) local degeneracies in the Voronoi tessellation picture, which represent an important extra difficulty when compared with the global comparison proved in \cite{Sev21} for instance. Indeed, the model considered in \cite{Sev21} had a crucially higher ``price of local disagrement'' than that of Voronoi percolation, thus only requiring rather crude control on the local degeneracies when performing the local surgery. For Voronoi percolation however, a lower price of local disagreement requires a  considerably more delicate control on the local degeneracies. For that purpose, we prove a large deviation result on the graph (or chemical) distance of Poisson Voronoi tessellation, which is interesting on its own --  see Proposition~\ref{prop:chemical_dist}.
An even more important difference with \cite{Sev21} is that we are able to prove local uniqueness and not only increasing properties coming directly from the global comparison statement, such as percolation on slabs. \\

\textbf{Organization of the paper.} In Section~\ref{sec:preliminaries} we prove a few results concerning the geometry of Poisson Voronoi tessellation, which will be useful in the subsequent sections. In Section~\ref{sec:comparison} we introduce a truncated version of Voronoi percolation and prove our global comparison result. In Section~\ref{sec:loc_uniq} we deduce local uniqueness from a strong bound on disconnection events obtained in Section~\ref{sec:comparison}. Finally, in Section~\ref{sec:Grim_Mars} we adapt the proof of Grimmett \& Marstrand to our truncated Voronoi percolation model. \\

\textbf{Notation.} Let $\|\cdot\|$ denote the $\ell^2$ (Euclidean) distance in $\R^d$ and $B_r(x)\coloneqq\{y\in\R^d:~\|y-x\|\leq r\}$ be the corresponding ball of radius $r\geq0$ centered at $x\in \R^d$ -- when $x=0$, we may simply write $B_r$. Given two sets $X,Y\subset \R^d$, we denote by $d(X,Y)\coloneqq \inf_{x\in X,y\in Y} \|x-y\|$ the distance between $X$ and $Y$. We may also consider the $\ell^\infty$ norm $\|\cdot\|_\infty$ and its associated distance $d_\infty$ and boxes $\Lc_r(x)\coloneqq x+\Lc_r=x+[-r,r]^d$. We say that two vertices $x,y\in \Z^d$ are $\ast$-linked if $\|x-y\|_\infty\leq 1$, which naturally induces a notion of $\ast$-connectivity in $\Z^d$.
Given a domain $D\subset \R^d$ and a random set $\pazocal{V}\subset \R^d$, we denote by $\{\lr{D}{\pazocal{V}}{X}{Y}\}$ the event that there exists a connected component of $\pazocal{V}\cap D$ intersecting both $X$ and $Y$, and by $\{\nlr{D}{\pazocal{V}}{X}{Y}\}$ its complement.

\section{Poisson--Voronoi tessellation}\label{sec:preliminaries}

In this section we prove a few basic properties of Poisson--Voronoi tessellation on $\R^d$. We collect here some technical results that will be used in the following sections. The reader can skip this section at first reading and refer to it when needed.

Let $\overline{\eta}$ be a Poisson point process of intensity $1$ on $\R^d$. Here we think of $\overline{\eta}$ as taking values in the space $\Omega$ of discrete subsets of $\R^d$, with the standard sigma-algebra.
The Voronoi tessellation associated to $\overline{\eta}$ is the pavement of $\R^d$ by the cells
\begin{equation*}
C(x)\coloneqq\{y\in\R^d:\, d(y,x)\leq d(y,\overline{\eta})\}, ~~~~~ x\in\overline{\eta}.
\end{equation*} 
In this way we obtain the random (embedded) graph $G=G(\overline{\eta})$ with vertex set $\overline{\eta}$ and edges between a pair of points $x,y\in \overline{\eta}$ if the cells $C(x)$ and $C(y)$ share a side.

We start by proving two basic lemmas concerning the size and number of cells intersecting a box.

\begin{lemma} \label{lemma:diam}There exist $C,c>0$ such that for any $L\ge 1$ and $\ell \ge C(\log L) ^ {1/d}$, 
  \begin{equation*}
      \P[\exists x\in \overline \eta : C(x)\cap \Lambda_L\ne \emptyset, \mathrm{diam}(C(x))\ge \ell ]\le e^ {-c\ell ^ d}.
  \end{equation*}
  \end{lemma}
\begin{proof}
    Assume the event $\pazocal F\coloneqq \{\exists x\in \overline \eta : C(x)\cap \Lambda_L\ne \emptyset, \mathrm {diam}(C(x))\ge \ell\}$ occurs and let $x\in\overline\eta$ be such a point.
    Let $y,z\in C(x)$ such that $\|z-y\|=\mathrm{diam} (C(x))\coloneqq k\ge \ell$. We claim that the event $\{\overline \eta \cap B_{k/3}(y)=\emptyset\}\cup \{\overline \eta \cap B_{k/3}(z)=\emptyset\}$ occurs. Indeed, if this event does not occur, then it would imply that $\|x-y\|\le k /3$ and $\|x-z\|\le k/3$ and contradict that  $\|y-z\|\ge k$.
    Hence,
\begin{equation*}
   \pazocal F\subset \bigcup _{k\ge \ell} \, \bigcup_{x\in \frac{k}{6\sqrt{d}}\Z^ d\cap \Lambda_{L+k}}\{(x+B_{k/6})\cap \overline \eta =\emptyset\}.
\end{equation*}
It yields that 
\begin{equation*}
    \P[\pazocal F]\le \sum_{k\ge \ell } (12\sqrt{d})^ d \left (\frac{L+k}{k}\right)^ de^ {-\alpha_dk^ d/6^d}\le e^ {-c\ell ^ d}
\end{equation*}
where we have used that $\ell \ge C(\log L) ^ {1/d}$ in the last inequality.
\end{proof}

\begin{lemma}\label{lemma:intersection boundary}
    There exist $C,c>0$ such that for all $L\ge 1$
    \begin{equation*}
        \P[|\{x\in\overline \eta : C(x)\cap \Lambda_L\}| \ge CL^ {d}]\le e^ {-cL^ d}.
    \end{equation*}
\end{lemma}
\begin{proof}
    Note that we have the following inclusion 
    \begin{equation*}
        \{|\{x\in\overline \eta : C(x)\cap \Lambda_L\}| \le 2^{d+1}L^ {d}\}\subset \{|\overline \eta \cap \Lambda_{2L}| \le 2^{d+1}L^ d\}\cap \{\forall x\in \overline \eta : C(x)\cap \Lambda_L\ne \emptyset, \mathrm{diam}(C(x))\le L \}.
    \end{equation*}
    We conclude by applying Lemma \ref{lemma:diam}.
\end{proof}

The following proposition states that, with probability exponentially close to one, the chemical distance in $G$ is comparable with the Euclidean distance. This result will be used in Section~\ref{sec:comparison}, where the exponent $M$ in \eqref{eq:chem_dist} below will be crucial.

\begin{proposition}\label{prop:chemical_dist}
    For $R,M \geq 2$, let $\pazocal{D}(R,M)$ be the event that for every pair of vertices $x,y\in \Lambda_R$ there exists a path of length at most $M$ in $G(\overline{\eta}_{|\Lambda_{2R}}\cup \omega)\cap \Lambda_{2R}$ between $x$ and $y$ for any point configuration $\omega$ on $\R^d\setminus \Lambda_{2R}$. 
    Then there exist constants $C_0,c_0\in(0,\infty)$ such that for all $R\geq 1$ and $C_0R\leq M\leq c_0 R^d$,
    \begin{equation}\label{eq:chem_dist}
        \P[\pazocal{D}(R,M)]\geq 1-e^{-c_0M}.
    \end{equation}
\end{proposition}
\begin{remark}
    It would perhaps be more natural to define $\pazocal{D}(R,M)$ replacing $G(\overline{\eta}\cup \omega)$ by $G(\overline{\eta})$. However, it will be convenient for us to work with a local event that depends only on $\overline{\eta}$ restricted to $\Lambda_{2R}$.
\end{remark}

\begin{proof}[Proof of Proposition~\ref{prop:chemical_dist}]
Fix $N_0>0$ large enough, and let $N_k\coloneqq 10^kN_0$. Consider the straight line $[x,y]$ between two points $x, y\in \Lambda_{R+1}$ and take a deterministic $N_0/2$-path $x=x_0,x_1,\cdots,x_n=y$ on $[x,y]$, i.e., $x_i\in[x,y]$ and $\|x_{i+1}-x_i\|\leq N_0/2$ for all $0\leq i\leq n-1$. Notice that we can take $n\leq 2\|x-y\|/N_0$. 
We say that $x_i$ is $k$-good if $\overline{\eta}\cap B_{N_k/2}(x_i)\neq \emptyset$ and $|\overline{\eta}\cap B_{N_k}(x_i)| \leq 2\mathrm{Vol}(B_{N_k})$; otherwise, we say that $x_i$ is $k$-bad. One can easily see that
\begin{itemize}
    \item there exists $a>0$ such that $\P[x_i \text{ is $k$-bad}]\leq e^{-aN_k^d}$ for all $i,k$,
    \item the events $\{x_i \text{ is $k$-bad}\}$ and $\{x_j \text{ is $l$-bad}\}$ are independent if $B_{N_k}(x_i)\cap B_{N_l}(x_j)=\emptyset$.
\end{itemize}
For each integer $1\leq i\leq n$, let $k_i\coloneqq\min\{k\geq 0: \text{ $x_i$ is $k$-good}\}$ and define the ``good zone'':
\begin{equation}\label{eq:Z_def}
    Z\coloneqq \bigcup_{i=0}^{n} B_{N_{k_i}}(x_i).
\end{equation}
By definition, one has $|Z\cap \overline{\eta}|\leq 2\mathrm{Vol}(Z)$ and the whole line $[x,y]$ is contained in the union of cells centered at $Z\cap \overline{\eta}$. Also notice that $Z\subset \Lambda_{2R}$ as long as $\mathrm{Vol}(Z)\leq c_0 R^d$ with $c_0$ small enough. The following deterministic lemma states that we can extract from $Z$ a region with comparable volume corresponding to the union of disjoint bad balls.
\begin{lemma}\label{lem:Z'_disjoint} 
    There exist constants $c_d>0$ and $C_0=C_0(N_0)<\infty$ such that the following holds. For every $Z$ as in \eqref{eq:Z_def} satisfying $\mathrm{Vol}(Z)\geq M\geq C_0 R$, there exist $I\subset\{0,1,\cdots,n\}$ such that $k_i\geq 1 ~\forall i\in I$, the balls $B_{N_{k_i-1}}(x_i)$, $i\in I$, are all disjoint, and $Z'\coloneqq \sqcup_{i\in I} B_{N_{k_i-1}(x_i)}$ satisfies $\mathrm{Vol}(Z')\geq c_dM$.
\end{lemma}
We delay the proof of Lemma~\ref{lem:Z'_disjoint} and continue with the proof of Proposition~\ref{prop:chemical_dist}.
Denote by $\alpha_d$ the volume of the unit ball in $\R^d$.
From this lemma, we can deduce that for $M\ge C_0 R$, we have
\begin{equation}
\begin{split}
   \P[\mathrm{Vol}(Z)\geq M]&\le \P\left[\begin{array}{c} \exists I\subset \{0,\dots,n\} \,\exists (n_i)_{i\in I} : \forall i\in I \quad x_i \text{ is $n_i$-bad }, \\\forall i\ne j \in I\quad B_{N_{n_i}}(x_i)\cap B_{N_{n_j}}(x_j)=\emptyset, \\\sum_{i\in I}\alpha_dN_{n_i}^d\ge c_d M \end{array}\right]\\
   &\le\sum_{k\ge \frac {c_d}{\alpha_d}M}\sum_{I\subset  \{0,\dots,n\} }\sum_{\substack{(n_i)_{i\in I}\\\text{ admissible}}}\P[ \forall i\in I \quad x_i \text{ is $n_i$-bad }]\\
   &\le \sum_{k\ge \frac {c_d}{\alpha_d}M}\sum_{I\subset  \{0,\dots,n\} }\sum_{ \substack{(n_i)_{i\in I}\\\text{ admissible}}} e^ {-a k}
   \end{split}
\end{equation}
where we say that $(n_i)_{i\in I}$ is admissible  for fixed $k$ and $I$ if for all $ i\ne j \in I$, $B_{N_{n_i}}(x_i)\cap B_{N_{n_j}}(x_j)=\emptyset$ and $ \sum_{i\in I}N_{n_i}^d=k$.
Clearly, it suffices to determine $N_{n_i}^d$ to determine $n_i$. Hence, the number of admissible sequences of  $(n_i)_{i\in I}$ is at most $k \choose{|I|-1}$. It is easy to check using Stirling's estimate that there exists a constant $\kappa$ depending on $a$ such that if $k\ge \kappa |I|$ then 
\begin{equation*}
    {k \choose{|I|-1}}\le e^ {\frac 1 2a k}.
\end{equation*}
Since $|I|\le n\le 4\sqrt d (R+1)/N_0$ and $k\ge \frac {c_d}{\alpha_d}M\ge\frac {c_d}{\alpha_d} C_0 R$, this condition is fulfilled as long as $C_0$ is large enough.
Combining the two previous inequalities, we conclude that
\begin{equation*}
    \P[\mathrm{Vol}(Z)\geq M]\le \sum_{k\ge \frac {c_d}{\alpha_d}M} 2^ n e^ {-\frac 1 2a k}\le e^ {-c_1 M}.
\end{equation*}
Finally, define for $x,y\in \Lambda_{R+1}$, $M\ge 1$ the event $\pazocal{E}(x,y,M)$ that there exists a path of length at most $M$ in $G(\overline{\eta}_{\Lambda_{2R}}\cup \omega)\cap \Lambda_{2R}$ between the cell of $x$ and $y$ for every point configuration $\omega$ on $\R^d\setminus \Lambda_{2R}$.
We have
\begin{equation*}
    \P[\pazocal{E}(x,y,M)^{\mathrm{c}}] \le \P[ Z\not\subset \Lambda_{2R}]+\P\left[\Vol(Z)\ge \frac M 2\right]\le e^{-c_1c_0R^ d}+ e^ {-c_1 \frac M2}\le e^ {-c_2M},
\end{equation*}
where we used in the second inequality that $Z\subset \Lambda_{2R}$ as long as $\mathrm{Vol}(Z)\leq c_0R^d$, and in the last inequality, we used that $M\le c_0 R^ d$.
Set $\ep \coloneqq e^ {-c_2\frac M{4d}}$ and define 
\[\pazocal{E}(M)\coloneqq \bigcap_{x,y\in \Lambda_{R+1}\cap (\ep \Z^ d)}\pazocal{E}(x,y,M) .\]
In particular, we have
\begin{equation}\label{eq: bound E(M)}
    \P[\pazocal{E}(M)^{\mathrm{c}}]\le (3R\ep^ {-1})^ {2d}\ep ^{4d}\le  3^{2d} R^{2d} e^{-c_2\frac M{2}}.
\end{equation}
Let $x\in \Lambda_R$ and assume that the cell of $x$ does not contain any point in $\ep \Z^ d$. Let $y\in \overline{\eta}$ be the center of the cell of $x$. The point $z\in \ep \Z^ d\cap \Lambda_{R+1}$ which is closest to $y$ is at distance at most $\sqrt d \ep $. Hence, if the cell of $y$ does not contain any point in $\ep \Z^ d$, it implies that there is another point of $\overline \eta$ at distance at most $\sqrt d \ep$ from $z$. It follows that if there exists $x\in \Lambda_R$ such that its corresponding cell does not contain any point in $\ep \Z^d$ then the following event occurs
\begin{equation*}
\pazocal{F}\coloneqq  \bigcup_{z\in\ep \Z ^ d\cap \Lambda_{R+1}}\{|\overline \eta \cap (B_{\sqrt d \ep}+z)| \ge 2\}.
\end{equation*}
Note that for $\lambda\geq0$ one has
$\P[\pazocal{P}(\lambda)\ge 2]= 1- e^ {-\lambda}- \lambda e^{-\lambda}\le \lambda^ 2$,
where $\pazocal{P} (\lambda)$ is a Poisson law of parameter $\lambda$.
This yields
\begin{equation}\label{eq: bound F}
    \P[\pazocal{F}]\le  (3R\ep^ {-1})^d \alpha_d^2(\sqrt d \ep)^ {2d}\le 3^{d}\alpha_d^2 d^d R^d e^ {-c_2\frac M{4}}.
\end{equation}
We conclude using inequalities \eqref{eq: bound E(M)} and \eqref{eq: bound F} together with the following inclusion
\begin{equation*}
    \pazocal{D}(M,R)\subset \pazocal{F}^{\mathrm{c}}\cap \pazocal{E}(M).
\end{equation*}
\end{proof}

\begin{proof}[Proof of Lemma~\ref{lem:Z'_disjoint}]
We will define $I$ inductively starting with the largest size of ball.
Set \[K\coloneqq\max_{0\le i \le n}k_i.\] Set $\overline I_{K+1}= I_{K+1}=\emptyset$.
Assume $I_{K+1},\dots,I_j$ and $\overline I_{K+1},\dots,\overline I_j$ have been defined for $2\le j\le K+1$.

Let $I_{j-1}\subset \{0,\dots,n\} $ be a maximal subset of  $\{0,\dots,n\}\setminus \cup_{j\le k \le K}\overline I_k$ such that $k_i=j-1$ for all $i\in I_{j-1}$, and $ B_{N_{j-1}(x_i)}\cap  B_{N_{j-1}(x_m)}=\emptyset$ for all $i, m\in I_{j-1} $ with $i\neq m$ (if there are several choices, we pick an arbitrary one). In particular, for every $i\in\{0,\dots,n\}\setminus \cup_{j\le k \le K}\overline I_k$ such that $k_i=j-1$, there exists $m\in I_{j-1}$ such that $B_{N_{j-1}(x_i)}\subset B_{3N_{j-1}(x_m)}$.
Define 
    \[Z_{j-1}\coloneqq \bigcup_{i\in I_{j-1}}B_{N_{j-1}(x_i)}.\] 
    Now, consider the following set of boxes intersecting the zone $Z_{j-1}$
    \[\overline I_{j-1}\coloneqq \{i\in\{0,\dots,n\}\setminus \cup_{j\le k \le K}\overline I_k:B_{N_{k_i}(x_i)}\cap Z_{j-1}\ne \emptyset \}.\]
and 
    \[\overline Z_{j-1}\coloneqq \bigcup_{i\in \overline I_{j-1}}B_{N_{k_i}(x_i)}.\]
It follows from the construction that $I_{j-1}\subset \overline{I}_{j-1}$ and $\{i\in\{0,\cdots,n\}:~k_i\geq j\} \subset \bigcup_{j\leq k \leq K} \overline{I}_k$. Also notice that, for every $i\in \overline{I}_{j-1}$ there exists $m\in I_{j-1}$ such that $B_{N_{k_i}}(x_i)\subset B_{N_{j-1}}(x_m)$. As a conclusion, we have
    \begin{equation*}
        \overline Z_{j-1}\subset \bigcup_{i\in I_{j-1}}B_{3N_{j-1}(x_i)}.
    \end{equation*}
    Since $(B_{N_{j-1}}(x_i))_{i\in I_{j-1}}$ are pairwise disjoint, we have
    \begin{equation*}
    \Vol(\overline Z_{j-1})\le \alpha_d(3N_{j-1})^ d|I_{j-1}|= 3^ d\Vol(Z_{j-1}).
\end{equation*}
Since $(Z_j)_{1\leq j\leq K}$ are pairwise disjoint, we have
\begin{equation}
    \Vol(\cup_{j=1}^ K\overline Z_j)\le  \sum_{j=1}^ K\Vol(\overline Z_j)\le 3^ d\sum_{j=1}^ K\Vol(Z_j)=3^d\Vol(\cup_{j=1}^ KZ_j).
\end{equation}
Finally, since $\{i\in\{0,\cdots,n\}:~k_i\geq 1\} \subset \bigcup_{1\leq k \leq K} \overline{I}_k$, we have
\begin{equation*}
      \Vol(\cup_{j=1}^ K\overline Z_j)\ge \Vol(Z)-\alpha_dnN_0^ {d}\ge \Vol(Z)-2\alpha_dN_0^ {d-1}R\ge \frac 1 2\Vol(Z)
\end{equation*}
where in the last inequality we used that $C_0\ge 2\alpha_dN_0^ {d-1}$.
We conclude by setting $I=\cup_{1\le j \le K}I_j$, $Z'\coloneqq \sqcup_{i\in I} B_{N_{k_i-1}(x_i)}$ and observing that
\begin{equation*}
    \Vol(Z')=\frac{1}{10^d}\Vol(\cup_{j=1}^{K} Z_j) \geq \frac{1}{30^d}\Vol(\cup_{j=1}^{K} \overline{Z}_j) \geq \frac 1 {2(30)^ d}\Vol(Z).
\end{equation*}
\end{proof}

We finish this section with the following technical lemma, which will be used in Section~\ref{sec:loc_uniq}. In words, it says that with very high probability, every large connect set in $B_L$ intersects a positive proportion of regions where the chemical distance is comparable with Euclidean distance.
Given $x\in\R^d$ and an event $\pazocal{A}$ on the space $\Omega$ of discrete subsets of $\R^d$, we denote by $x+\pazocal{A}$ the event consisting of configurations of the form $x+\omega\coloneqq \{x+z,~z\in \omega\}$ for $\omega\in \pazocal{A}$.

\begin{lemma}\label{lemma :GL}There exist $\kappa,c>0$ such that for any $L\ge \ell\ge 1$ large enough, we have
\begin{equation*}
    \P[\cG_{L,\ell,\kappa}]\le e^ {-c\sqrt L}
\end{equation*}
where
    \begin{equation*}
        \cG _{L,\ell,\kappa} \coloneqq  \left\{ \begin{array}{c} \forall\, \Gamma \subset \Z^d \text{ $\ast$-connected set intersecting $\Lambda_{L/\ell}$ such that $|\Gamma|\ge \frac {\sqrt L} {1000 \ell}$,}\\
        \big| \{x\in \Gamma : \ell x+\pazocal{D}(3\ell,\kappa\ell )\text { occurs}\} \big| \ge \frac {|\Gamma|} 2\end{array}\right\}.
    \end{equation*}
    
\end{lemma}

\begin{proof}
    Let $\kappa\ge C_0>0$ be a constant large enough.
    Note that if $x,y \in \Z^ d$ are such that $\|x-y\|\geq 12$ then the events $\ell x+\pazocal{D}(3\ell,\kappa\ell )$ and $\ell y+\pazocal{D}(3\ell,\kappa\ell )$ are independent.
    Let $\Gamma$ be a connected set in $\Z^d$. If $|\{x\in \Gamma : \ell x+\pazocal{D}(3\ell,\kappa\ell )\text { occurs}\}| < |\Gamma|/2$, then one can find a subset $\Gamma'\subset \Gamma$ such that $\|x-y\|\geq 12$ for all distinct  $x,y\in \Gamma'$,  $\ell x+\pazocal{D}(3\ell,\kappa\ell )$ does not occur for all $x\in \Gamma'$ and $|\Gamma'|\geq c|\Gamma|$, with $c=c(d)>0$ depending only on $d$.
    By Proposition \ref{prop:chemical_dist}, it yields that
\begin{equation*}
\begin{split}
   \P\left[ \big| \{x\in \Gamma : \ell x+\pazocal{D}(3\ell,\kappa\ell )\text { occurs}\} \big| < \frac{|\Gamma|}{2}\right] \le 2^{|\Gamma|} \P[ \pazocal{D}(3\ell,\kappa\ell )^{\rm c}]^ {c|\Gamma|} \le e^ {-c\kappa \ell |\Gamma|}.
\end{split}
\end{equation*}
    Since the number of $\ast$-connected subsets of $\Z^d$ of size $k$ containing the origin is at most $7^{dk}$ \cite[(4.24) p81]{Gri99a}, a union bound over all possible set $\Gamma$ gives 
\begin{equation}\label{ineq:G}
    \P[\cG_{L,\ell,\kappa} ^ {\rm c}]\le (2L)^ d\sum_{k\ge \frac {\sqrt L} {1000 \ell}}7^{dk} e^ {-c\kappa\ell k}\le e^ {-c \sqrt L}
\end{equation}
    where we used that $\kappa$ is large enough. 
\end{proof}

\section{Global comparison}\label{sec:comparison}

In this section we compare the Voronoi percolation model $\pazocal{V}$ with its truncated version $\pazocal{V}_N$, for which supercritical sharpness is easier to prove. 

We start by introducing a convenient coupling of $\pazocal{V}(p)$, $p\in[0,1]$. 
Let $\eta$ be Poisson point process of intensity $1$ on $\R^d\times [0,1]$. We denote the points in $\R^d\times[0,1]$ by $(x,t)$, with $x\in\R^d$ and $t\in [0,1]$. For each function $\fp:\R^d\to [0,1]$, define the set of $\fp$-open and $\fp$-closed points given by 
\begin{align*}
\overline{\eta}_o(\fp)&\coloneqq\{x\in \R^d:\, \exists t\in[0,\fp(x)] \text{ such that } (x,t)\in \eta\} ~~\text{ and}\\ 
\overline{\eta}_c(\fp)&\coloneqq\{x\in \R^d:\, \exists t\in(\fp(x),1] \text{ such that } (x,t)\in \eta\}
\end{align*}
When $\fp$ is the constant function equal to $p\in[0,1]$, we simply write $\overline{\eta}_o(p)$ and $\overline{\eta}_c(p)$.
We can now set
\begin{equation*}
\pazocal{V}(p)\coloneqq\{y\in \R^d:~ d(y,\overline{\eta}_o(p))\leq d(y,\overline{\eta}_c(p))\}.
\end{equation*}
Notice that conditionally on $\overline{\eta}\coloneqq \overline{\eta}_o(1)$, this corresponds to the standard coupling of Bernoulli site percolation on the graph $G(\overline{\eta})$ defined in Section~\ref{sec:preliminaries}.
Furtheremore, $\pazocal{V}(p)$ is increasing in $\eta_o(p)$ and decreasing in $\eta_c(p)$, which are independent Poisson point processes for any fixed $p$, thus implying that the model satisfies the FKG inequality, see e.g.~\cite[Theorem 20.4]{last_penrose_2017}.

We now construct the truncated model using the same Poisson point process $\eta$. For every $N\geq1$ and $x\in N\Z^d$, we define $\Lu_N(x)\coloneqq x+[0,N)^d$ -- note that these boxes perfectly pave $\R^d$. Given $\fp:\R^d\to [0,1]$, we define $\overline{\eta}_o^N(\fp) \coloneqq \overline{\eta}_o(\fp)$ and, for every $x\in N\Z^d$,
\begin{align}\label{eq:def_eta_N}
\overline{\eta}_c^N(\fp)\cap \Lu_N(x) &\coloneqq
\begin{cases}
    \overline{\eta}_c(\fp)\cap \Lu_N(x),    & \text{ if } |\overline{\eta}_c(\fp)\cap \Lu_N(x)|\leq 2N^{d},\\
    \Lu_N(x),     &  \text{ otherwise.}
\end{cases}
\end{align}
In words, $(\overline{\eta}^N_o(\fp),\overline{\eta}^N_c(\fp))$ is equal $(\overline{\eta}_o(\fp),\overline{\eta}_c(\fp))$ on $N$-boxes with at most $2N^{d}$ many closed points, while the other boxes are considered fully closed.
We can now define our truncated model as
\begin{equation}\label{eq:def_truncated_Voronoi}
\pazocal{V}_N(p)\coloneqq\big\{y\in \R^d:~ d(y,\overline{\eta}_o^N(p))\leq \min\{d(y,\overline{\eta}^N_c(p)), N\} \big\}.
\end{equation}
Notice that even though the boxes of the form $\Lu_N(x)$, $x\in N\Z^d$, which are the building blocks in the definition of $\pazocal{V}_N(p)$, are not exactly symmetric due to their boundaries, almost surely all points in the Poisson point process $\overline{\eta}$ fall in the interior of these boxes. In particular, our model $\pazocal{V}_N(p)$ is invariant under all the symmetries of $N\Z^d$. One can easily check that, for every $D\subset \R^d$, the set $\pazocal{V}(p)\cap D$ depends only on the restriction of $\eta$ to $(D+\Lc_{2N}(y))\times[0,1]$.
In particular, $\pazocal{I}_k$ is independent in regions with pairwise $\ell^\infty$ distance at least $4N$. 
Furthermore, $\pazocal{V}_N(p)$ is increasing in $\overline{\eta}_o(p)$ and decreasing in $\overline{\eta}_c(p)$, thus it also satisfies the FKG inequality.
Finally, the bound on the number of closed points in a box implies that $\pazocal{V}_N(p)$ satisfies a sort of sprinkled finite energy property -- see Proposition~\ref{prop:domsto}. 

The features of our truncated model $\pazocal{V}_N(p)$ pointed out above allow us to adapt several classical techniques from Bernoulli percolation. 
First, it follows from standard arguments that the model has a non-trivial phase transition, namely $p_c(N)\in(0,1)$, where 
\begin{equation*}
p_c(N)\coloneqq\inf\big\{p\in \R:~ \P[\pazocal{V}_N(p) \text{ percolates}]>0\big\}.
\end{equation*}
More importantly, we are able to adapt the proof of \cite{GriMar90} to our truncated model.
In particular, we prove the following result, whose proof is presented in Section~\ref{sec:Grim_Mars}.

\begin{theorem}\label{thm:disc_decay_truncated}
For every $N\geq1$ and $p>p_c(N)$, there exists $c=c(p,N)>0$ such that for $R$ large enough,
\begin{equation}\label{eq:disc_decay_truncated} 
    \P[\lr{}{\pazocal{V}_N(p)}{B_R}{\infty}]\ge 1-e^ {-cR^ {d-1}}.
\end{equation}
\end{theorem}
\begin{remark}\label{rem:sharpness_truncated}
    We can in fact prove more about $\pazocal{V}_N$, i.e. it satisfies the properties stated in Theorems~\ref{thm:sharpness} and \ref{thm:unique}. However, we will only need \eqref{eq:disc_decay_truncated} -- see also Remarks~\ref{rem:comparison_extensions} and \ref{rem:decay_extensions} below.
\end{remark}

Notice that by definition $\pazocal{V}_N(p)\subset \pazocal{V}(p)$ almost surely.
Hence, the bound \eqref{eq:disc_decay_truncated} from Theorem~\ref{thm:disc_decay_truncated} also holds with $\pazocal{V}_N(p)$ replaced by $\pazocal{V}(p)$ for all $p>\inf_N p_c(N)\geq p_c$.
The following result, which is the heart of our proof, guarantees that the truncated model $\pazocal{V}_N$ approximates well the original model $\pazocal{V}$ globally, provided we use a small change in the parameter $p$. In particular, this allows us to transfer Theorem~\ref{thm:disc_decay_truncated} from $\pazocal{V}_N$ to $\pazocal{V}$ throughout the supercritical regime.

\begin{theorem}[Global comparison]\label{thm:comparison}
For every $\varepsilon>0$ there exists $N_0=N_0(\varepsilon)\in(0,\infty)$ such that for every $N\geq N_0$, every $p\in[2\varepsilon,1-2\varepsilon]$ and every  connection event $A=\{\lr{}{}{B_r}{\partial B_R}\}$, with $1\leq r<R<\infty$, one has
\begin{equation*}
    \P[\pazocal{V}(p)\in A]\leq \P[\pazocal{V}_N(p+\varepsilon)\in A].
\end{equation*}
In particular, one has $\lim_{N\to\infty} p_c(N)=p_c$.
\end{theorem}
\begin{remark}\label{rem:comparison_extensions}
    Our proof actually shows that one can take $N_0(\varepsilon)=C_\delta(\log \varepsilon^{-1})^{1+\delta}$ for any $\delta>0$. Furthermore, it is straightforward to adapt the proof to more general connection events of the form $A=\{\lr{D}{}{S_1}{S_2}\}$, as long as there is a ``uniformly flat surface'' in $D$ separating $S_1$ and $S_2$, see \cite[Remark 3.6]{Sev21}. In particular, one could consider connection events between balls restricted to 2D slabs.
\end{remark}

Theorems~\ref{thm:comparison} and \ref{thm:disc_decay_truncated} have the following direct consequence.

\begin{corollary}[Decay of disconnection]
\label{cor:decay_disconnection}
For every $p>p_c$, there exists $c=c(p)>0$ such that for $R$ large enough,
\begin{equation}\label{eq:decay_disconection}
    \P[\lr{}{\pazocal{V}(p)}{B_R}{\infty}]\ge 1-e^ {-cR^ {d-1}}.
\end{equation}
\end{corollary}
\begin{remark}\label{rem:decay_extensions}
    As mentioned in Remark~\ref{rem:sharpness_truncated}, the proof of Theorem~\ref{thm:disc_decay_truncated} in Section~\ref{sec:Grim_Mars} also implies that $\pazocal{V}_N(p)$ percolates on slabs for all $p>p_c(N)$, $N\geq1$. Therefore, the second part of Theorem~\ref{thm:sharpness} follows directly from this fact together with the observation made in Remark~\ref{rem:comparison_extensions}. The events in \eqref{eq:exp_decay_diam}, \eqref{eq:exp_decay_vol} and \eqref{eq:unique} on the other hand are non-monotonic (combination of) connection events, and therefore bounds on their probabilities cannot be transferred directly from $\pazocal{V}_N$ to $\pazocal{V}$ by using Theorem~\ref{thm:comparison}. In Section~\ref{sec:loc_uniq}, we will deduce \eqref{eq:exp_decay_diam}, \eqref{eq:exp_decay_vol} and \eqref{eq:unique} from the decay of disconnection in \eqref{eq:decay_disconection}.
\end{remark}

\subsection{Interpolation scheme}
We now turn to the proof of Theorem~\ref{thm:comparison}, which is inspired by \cite{Sev21}.
Fix any $\delta\in(0,1)$ and set
\begin{equation}\label{eq:choice_epsilon}
\varepsilon_{\sN}\coloneqq e^{-N^{1-\delta}}.
\end{equation}
Theorem~\ref{thm:comparison} is a direct consequence of the following proposition.
\begin{proposition}\label{prop:comp_N_2N}
For every $\varepsilon>0$ there exists $N_1=N_1(\varepsilon,\delta)\geq1$ such that the following holds. For every $N\geq N_1$, every $p\in[\varepsilon,1-\varepsilon]$ and every connection event $A=\{\lr{}{}{B_r}{\partial B_R}\}$, with $1\leq r<R<\infty$, one has
\begin{equation}\label{eq:comp_N_2N}
    \P[\pazocal{V}_{2N}(p)\in A]\leq \P[\pazocal{V}_{N}(p+\varepsilon_{\sN})\in A].
\end{equation}
\end{proposition}

\begin{proof}[Proof of Theorem~\ref{thm:comparison}]
By definition, $\pazocal{V}_N(p)$ coincides locally with $\pazocal{V}(p)$ for sufficiently large $N$, therefore $\lim_{N\to\infty} \P[\pazocal{V}_{N}(p)\in A]=\P[\pazocal{V}(p)\in A]$. By successively applying Proposition~\ref{prop:comp_N_2N} to $2^iN$, $i\geq 0$, and noticing that $\sum_{i\geq 0} \varepsilon_\si{2^i N}\asymp e^{-N^{1-\delta}}< \varepsilon$ for $N$ sufficiently large (depending on $\varepsilon$), the theorem follows.
\end{proof}

\begin{proof}[Proof of Proposition~\ref{prop:comp_N_2N}]

We start by fixing $\varepsilon>0$, $N\geq1$ and $p\in[2\varepsilon,1-2\varepsilon]$. We will construct an interpolation between $\pazocal{V}_{2N}(p)$ and $\pazocal{V}_{N}(p+\varepsilon_{\sN})$ as follows. Let $\{x_0,x_1,x_2,\dots\}$ be an arbitrary enumeration of $2N\Z^d$. We start with $\pazocal{V}_{2N}(p)$ and at the $n$-th step of our procedure we modify the model around the box $\Lu_{2N}(x_n)=x_n+[0,2N)^d$.
Each such step is decomposed into two ``half-steps'', where we first change the model from $\pazocal{V}_{2N}$ to $\pazocal{V}_{N}$ inside $\Lu_{2N}(x_n)$, and then sprinkle by a decaying profile function around that box.
We do so in such a way that at the ``$\infty$-th step'' of our procedure, we end up with $\pazocal{V}_{N}(p+\varepsilon_{\sN})$. Note that each $2N$-box $\Lu_{2N}(x)$, $x\in 2N\Z^d$ is exactly the union of $2^d$ many $N$-boxes $\Lu_{N}(x')$, $x'\in N\Z^d$.

We now describe the precise construction. 
Since $N$ and $p$ are fixed once and for all, we may omit them from the notation when constructing the interpolation.
First, let $\tau:\Z^d\to\R$ be the function given by
\begin{equation*}
    \tau(x)\coloneqq c_d(1+\|x\|^{d+1})^{-1}, ~~~
    x\in \Z^d,
\end{equation*}
where the constant $c_d>0$ is chosen so that $\sum_{x\in\Z^d} c_d(1+\|x\|^{d+1})^{-1}=1$.
Consider the increasing sequence $(\tau_k)_{k\in\frac12 \N}$ of ``sprinkling functions'' from $\R^d$ to $\R$ recursively defined by $\tau_0\equiv 0$ and, for all integers $n,i\geq 0$,
\begin{align}
    \label{eq:tau_k_1}
    \tau_{n+\frac12}(x_i) &\coloneqq  \tau_n(x_i),\\
    \label{eq:tau_k_2}
    \tau_{n+1}(x_i) &\coloneqq  \tau_{n+\frac12}(x_i) + \varepsilon_{\sN} \tau \big(\tfrac{x_i-x_n}{2N}\big),
\end{align}
and $\tau_k(y)=\tau_k(x_i)$ for all $y\in \Lu_{2N}(x_i)$, and all $k,i\geq0$.
Notice that by construction, the limit $\tau_{\infty}\coloneqq \lim_{n\to\infty} \tau_n$ is the constant function equal to $\varepsilon_{\sN}$. 
We are now ready to construct the interpolation. For every $k\in\frac12\N$, $a\in\{o,c\}$ and $i\geq0$, we define
\begin{equation*}
    \omega^k_a\cap \Lu_{2N}(x_i)\coloneqq
    \begin{cases}
        \overline{\eta}_{a}^{N}(p+\tau_k)\cap \Lu_{2N}    (x_i),    &\text{ for } i\leq \lceil k \rceil-1,\\
        \overline{\eta}_a^{2N}(p+\tau_k)\cap \Lu_{2N}     (x_i)   &\text{ for } i\geq \lceil k \rceil.
    \end{cases}
\end{equation*}
For every $k\in\frac12\N$, set
\begin{align*}
    \pazocal{I}^1_k &\coloneqq  \Big\{y\in\bigcup_{i=0}^{\lceil k\rceil-1} \Lu_{2N}(x_i):~ d(y,\omega^k_o)\leq\min\{d(y,\omega^k_c),\, N\} \Big\},\\
    \pazocal{I}^2_k &\coloneqq \Big\{y\in\bigcup_{i=\lceil k\rceil}^{\infty} \Lu_{2N}(x_i):~ d(y,\omega^k_o)\leq\min\{d(y,\omega^k_c),\, 2N\}\Big\},
\end{align*}
and finally define
\begin{equation*}
    \pazocal{I}_k \coloneqq  \pazocal{I}^1_k\cup \pazocal{I}^2_k.
\end{equation*}
The sequence $(\pazocal{I}_k)_{k\in\frac12 \N}$ is an interpolation between $\pazocal{V}_{2N}(p)$ and $\pazocal{V}_{N}(p+\varepsilon_{\sN})$.
Indeed, by construction $\pazocal{I}_0=\pazocal{I}_0^2=\pazocal{V}_{2N}(p)$ and, since $\tau_{\infty}\equiv \varepsilon_{\sN}$, we have
\begin{equation*}
    \pazocal{I}_\infty\coloneqq \lim_{k\to\infty} \pazocal{I}_k =     \lim_{k\to\infty} \pazocal{I}^1_k =  \pazocal{V}_{N}    (p+\varepsilon_{\sN}).
\end{equation*}
In words, for every $n\in\N$, we construct $\pazocal{I}_{n+\frac12}$ from $\pazocal{I}_n$ by changing the model in $\Lu_{2N}(x_n)$ from $\pazocal{V}_{2N}$ to $\pazocal{V}_{N}$; and we construct $\pazocal{I}_{n+1}$ from $\pazocal{I}_{n+\frac12}$ by sprinkling everywhere by an integrable function $\varepsilon_{\sN}\tau$ of the (renormalized) distance to the ``center'' $x_n$ -- see \eqref{eq:tau_k_2}. Notice that $\pazocal{I}_k$ is increasing in $\eta_o(p+\tau_k)$ and decreasing in $\eta_c(p+\tau_k)$, and therefore it satisfies the FKG inequality. We will repeatedly use the fact that for every $D\subset \R^d$, the set $\pazocal{I}_{k}\cap D$ depends only on the restriction of $\eta$ to $(D+\Lc_{4N}(y))\times[0,1]$.\footnote{ If $D$ is of the form $D=\cup_{j\in J} \Lu_{2N}(x_j)$, then $\pazocal{I}_{k}\cap D$ depends only on the restriction of $\eta$ to $(D+\Lc_{2N}(y))\times[0,1]$.} 
In particular, $\pazocal{I}_k$ is independent in regions with pairwise $\ell^\infty$ distance at least $8N$.

The crucial property of this construction is that it is ``almost increasing'' in $k$. First, we obviously have $\pazocal{I}_{n+\frac12}\subset \pazocal{I}_{n+1}$ almost surely for every $n\in\N$. Second, $\pazocal{I}_{n}= \pazocal{I}_{n+\frac12}$ holds with very high probability. Indeed, notice that by construction one has $\pazocal{I}_{n}= \pazocal{I}_{n+\frac12}$ as long as $|\overline{\eta}_c(1)\cap \Lu_{N}(x')|\leq 2N^{d}$ for all $x'\in N\Z^d$ with $\Lu_N(x')\subset \Lu_{2N}(x_n)$, and $d(y,\overline{\eta}_o(p))\leq N$ for all $y\in \Lu_{2N}(x_n)$.
By standard bounds on Poisson random variables, there exists $c=c(\varepsilon)>0$ such that 
\begin{equation}\label{eq:n_n+1/2}	
    \P[\pazocal{I}_{n} = \pazocal{I}_{n+\frac12}]\geq 1-e^{-cN^d}.
\end{equation}
We claim that, if $N$ is large enough, then for every connection event $A$, we have
\begin{equation}\label{eq:n_n+1}
    \P[\pazocal{I}_n\in A] \leq \P[\pazocal{I}_{n+1}\in A] ~~~ \text{for all } n\geq0.
\end{equation}
We now proceed with the proof of \eqref{eq:n_n+1}, which readily implies the desired inequality \eqref{eq:comp_N_2N}. In what follows, $N\geq1$ and $A=\{\lr{}{}{B_r}{\partial B_R}\}$ are fixed and thus often omitted from the notation, but every estimate will be uniform on them. Let
\begin{align}
    \label{eq:def_p_n}
    &p_n\coloneqq \P[\pazocal{I}_n\in A]-\P[\pazocal{I}_{n+\frac12}\in A]\\
    \label{eq:def_q_n}
    &q_n\coloneqq \P[\pazocal{I}_{n+1}\in A]-\P[\pazocal{I}_{n+\frac12}\in A] = \P[\pazocal{I}_{n+1}\in A,~\pazocal{I}_{n+\frac12}\notin A] ~~~(\geq0).
\end{align}
Our goal is to prove that, if $N$ is large enough (not depending on $A$), then $p_n\leq q_n$ for every $n\geq0$. 
Given $n\geq0$, $y\in\R^d$ and $L\geq0$, we define the \emph{coarse pivotality} event
\begin{equation*}
    \text{Piv}^{n}_{y}(L)\coloneqq \{\pazocal{I}_{n+\frac12}\cup \Lc_{L}(y) \in A\}\cap \{\pazocal{I}_{n+\frac12}\setminus \Lc_{L}(y) \notin A\}.
\end{equation*}
Notice that $\text{Piv}^{n}_{y}(L)$ depends only on $\pazocal{I}_{n+\frac12}$ restricted to the complement of $\Lc_L(y)$. Since $\pazocal{I}_{n}\cap \Lc_{4N}(x_n)^\mathrm{c}= \pazocal{I}_{n+\frac12}\cap \Lc_{4N}(x_n)^\mathrm{c}$ almost surely, we have the inclusions
\begin{align*} 
    \{\pazocal{I}_n\in A,~ \pazocal{I}_{n+\frac12}\notin A\} &\subset \{\pazocal{I}_{n}\neq\pazocal{I}_{n+\frac12}\} \cap \text{Piv}^{n}_{x_n}(4N) \\ 
    &\subset \{\pazocal{I}_{n}\neq                         \pazocal{I}_{n+\frac12}\} \cap \text{Piv}^{n}_{x_n}(16N).
\end{align*}
The event $\{\pazocal{I}_{n}\neq \pazocal{I}_{n+\frac12}\}$, besides satisfying the bound \eqref{eq:n_n+1/2}, only depends on $\eta$ restricted to $\Lc_{4N}(x_n)\times[0,1]$, which in turn is independent of $\pazocal{I}_{n+\frac12}$ restricted to $\Lc_{16N}(x_n)^{\mathsf{c}}$. Combining these observations, we obtain
\begin{align}\label{eq:decoupling}
    \begin{split}
        p_n&\leq\P[\pazocal{I}_n\in A,~   
        \pazocal{I}_{n+\frac12}\notin A] \\
        &\leq \P[\pazocal{I}_{n}\neq                     \pazocal{I}_{n+\frac12},~ \text{Piv}^{n}_{x_n}(16N)] \leq e^{-cN^d} \P[\text{Piv}^{n}_{x_n}(16N)].
    \end{split}
\end{align}
For $n,j\geq0$, let 
\begin{equation*}
    p_n(x_j)\coloneqq \P[\text{Piv}^{n}_{x_j}(16N)].
\end{equation*}
In view of \eqref{eq:decoupling}, it remains to show that if $N$ is large enough, then $p_n(x_n)\leq e^{cN^d}q_n$ for every event $A$ and all $n\geq0$. In other words, we want to construct the \emph{``sprinkling pivotality''} event $\{\pazocal{I}_{n+1}\in A,~\pazocal{I}_{n+\frac12}\notin A\}$ -- which has probability $q_n$, see \eqref{eq:def_q_n} -- out of the coarse pivotality event $\text{Piv}^n_{x_n}(16N)$ -- which has probability $p_n(x_n)$ -- by paying a price which is less than exponential in the volume. This fact is a straightforward consequence of the following lemma. Roughly speaking, it says that if coarse pivotality happens at a given site $x_j$, then one can either perform a \emph{``local surgery''} -- whose cost depends on the sprinkling function $\tau$ -- to construct sprinkling pivotality, or a local bad event -- which has a very small probability -- happens around $x_j$ and one further recovers a coarse pivotality event at some site $x_{j'}$ near $x_j$.

\begin{lemma}[Local surgery]\label{lem:conv_bound}
    There exist constants $N_1\geq1$ and $c,C\in(0,\infty)$ such that for $N\geq N_1$, every connection event $A$ and every $n,j\geq0$, 
    \begin{equation}\label{eq:conv_bound}
        p_n(x_j)\leq e^{CN^{\gamma'}}\, \tau\big(\tfrac{x_j-x_n}{2N}\big)^{-C N^\gamma} \,q_n ~+~ e^{-cN^{\gamma}} \sum_{x_{j'}\in \Lc_{20N}(x_j)} p_n(x_{j'}),
    \end{equation}
    where $\gamma=d-1+\delta/2$ and $\gamma'=d- \delta/2$.
\end{lemma}

Before proving Lemma~\ref{lem:conv_bound}, let us conclude the proof that $p_n\leq q_n$ for every $n\geq0$. Starting with $j=n$ and then successively applying Lemma~\ref{lem:conv_bound}, one can easily prove by induction on $T\geq1$ that the following holds
\begin{align*}
    \begin{split}
        p_n(x_n)
        & \leq e^{CN^{\gamma'}} \Big(\sum_{t=0}^{T-1} e^{-cN^{\gamma}t}\, |2N\Z^d\cap \Lc_{10N}|^t\, |2N\Z^d\cap \Lc_{20Nt}|\, \big(c_d^{-1}+c_d^{-1}(20t)^{d+1}\big)^{CN^\gamma}\Big)\,q_n \\
        & +e^{-cN^{\gamma}T} |2N\Z^d\cap \Lc_{20N}|^T \sum_{x_j\in \Lc_{20NT}(x_n)} p_n(x_j).
    \end{split}
\end{align*}  
Noting that the sum inside the parenthesis above is smaller than $e^{C'N^\gamma}$ and that the second term vanishes when $T\to\infty$ (recall that $p_n(x_i)\leq1$), we obtain
\begin{equation*}
    p_n(x_n)\leq e^{C'N^{\gamma'}}\,q_n,
\end{equation*}
which combined with \eqref{eq:decoupling} and the fact that $\gamma'<d$, implies that for $N$ large enough we have $p_n(x_n)\leq q_n$ for every connection event $A$ and every $n\geq0$, as we wanted to prove.
\end{proof}

It remains to prove Lemma~\ref{lem:conv_bound}, which is the technical heart of our proof.

\begin{proof}[Proof of Lemma~\ref{lem:conv_bound}]	
As described above, we want to construct the sprinkling pivotality event $\{\pazocal{I}_{n+1}\in A,~\pazocal{I}_{n+\frac12}\notin A\}$ starting from the coarse pivotality event $\text{Piv}^n_{x_n}(16N)$. We do so in four steps. In the first step, we ``complete the dual surface'' induced by coarse pivotality. In the second step, we decrease the distance between the two almost touching clusters. In the third step, we reduce to the situation where a certain ``good event'' happens around the touching region. In the fourth and final step, we use this good event to obtain the desired sprinkling pivotality.

Fix $N\geq1$ and $A\coloneqq \{\lr{}{}{B_r}{\partial B_R}\}$ a connection event. We stress that every estimate below will be uniform on $N$ and $A$. For $K\subset \R^d$, we will denote by $\eta_{|K}$ the restriction of $\eta$ to $K\times[0,1]$. 
Given $\fp:\R^d\to[0,1]$, we say that an event is $\fp$-increasing (resp.~decreasing) in $\eta_{|K}$ if it is increasing (resp.~decreasing) in $\eta_c(\fp)$ restricted to $K$ and decreasing (resp.~increasing) in $\eta_c(\fp)$ restricted to $K$. Given $n\geq0$, $L \geq 0$ and $y\in\R^d$, we define the \emph{closed pivotality} event
\begin{equation*}
    \text{CPiv}^n_y(L)\coloneqq \{\pazocal{I}_{n+\frac12}\cup \Lc_{L}(y) \in A\}\cap \{\pazocal{I}_{n+\frac12} \notin A\}.
\end{equation*}
In words, $\Lc_L(y)$ is called closed pivotal if it is pivotal but $A$ does not happen. The proof will be divided into four steps. \\
\vspace{0cm}

\textbf{Step 1:} \textit{From $16N$-pivotal to $20N$-closed-pivotal.}\\

\noindent
Let $n,j\geq0$. Consider the following event
\begin{equation*}
    \text{Piv}^n_{x_j}(20N,16N)\coloneqq \{\pazocal{I}_{n+\frac12}\cup \Lc_{20N}(x_j) \in A\}\cap \{\pazocal{I}_{n+\frac12}\setminus \Lc_{16N}(x_j) \notin A\}.
\end{equation*}
By definition $\text{Piv}^n_{x_j}(16N)\subset \text{Piv}^n_{x_j}(20N,16N)$, hence
\begin{equation}\label{eq:inclusion_piv}
    p_n(x_j)=\P[\text{Piv}^n_{x_j}(16N)]\leq \P[\text{Piv}^n_{x_j}(20N,16N)]. 
\end{equation}
Now, consider the event 
\begin{equation*}
    F\coloneqq \{\pazocal{I}_{n+\frac12}\cap \partial \Lc_{16N}(x_j)=\emptyset\} \cap \{\pazocal{I}_{n+\frac12} \cap \Lc_{16N}(x_j) \cap \partial B_r =\emptyset\}.
\end{equation*}
Now notice that (see Figure~\ref{surgery}) 
\begin{equation*}
    \text{Piv}^n_{x_j}(20N,16N)\cap F \subset  \text{CPiv}^n_{x_j}(20N).
\end{equation*} 
Furthermore, both events $\text{Piv}^n_{x_j}(20N,16N)$ and $F$ are $(p+\tau_{n+\frac12})$-decreasing in $\eta_{|\Lc_{18N}(x_j)}$ and $F$ is $\eta_{|\Lc_{18N}(x_j)}$-measurable. Therefore, conditioning on $\eta_{|\Lc_{18N}(x_j)^{\mathsf{c}}}$ and applying the FKG inequality for $\eta_{|\Lc_{18N}(x_j)}$ gives 
\begin{equation}\label{eq:FKG_piv_F}
    \P[\text{CPiv}^n_{x_j}(20N)]\geq \P[\text{Piv}^n_{x_j}(20N,16N)\cap F] \geq \P[\text{Piv}^n_{x_j}(20N,16N)]\P[F].
\end{equation}
We claim that
\begin{equation}\label{eq:bound_F}
    \P[F]\geq e^{-CN^{d-1}}.
\end{equation}
To prove \eqref{eq:bound_F}, first notice that $\P[B_1(x)\cap \pazocal{I}_{n+\frac12}=\emptyset]\geq c=c(\varepsilon)>0$ for all $x\in\R^d$. Then \eqref{eq:bound_F} follows by covering $\partial \Lc_{16N}$ and $\Lc_{16N}(x_j) \cap \partial B_r$ with $O(N^{d-1})$ many balls of radius $1$ and using the FKG inequality.
The inequalities \eqref{eq:inclusion_piv}, \eqref{eq:FKG_piv_F} and \eqref{eq:bound_F} together give
\begin{equation}\label{eq:4Npiv_to_8Ncpiv}
    p_n(x_j)\leq e^{CN^{d-1}}\P[\text{CPiv}^n_{x_j}(20N)].
\end{equation}
\vspace{0cm}

\textbf{Step 2:} \textit{From $20N$-closed-pivotal to $6N$-closed-pivotal.}\\

\noindent
First, let $\mathcal{C}_{r}$ and $\mathcal{C}_{R}$ denote the (union of) connected components of $\pazocal{I}_{n+\frac12}$ intersecting $B_r$ and $\partial B_R$, respectively. For $y\in\Z^d$, consider the events
\begin{align*}
    &E_1\coloneqq  \{\mathcal{C}_{r}\cap \partial B_R=\emptyset\}\cap\{\mathcal{C}_{r}\cap \Lc_{20N}(x_j)\neq\emptyset\}\\
    &E_2(y)\coloneqq \{\mathcal{C}_{R}\cap \Lc_{1}(y)\neq\emptyset\},
\end{align*}
and notice that, for $E(y)\coloneqq E_1\cap E_2(y)$, we have
$$\text{CPiv}^n_{x_j}(20N)\subset\bigcup_{y\in \Lc_{20N}(x_j)\cap \Z^d} E(y),$$
By a union bound, we can find $y_0\in \Lc_{20N}(x_j)\cap \Z^d$ such that
\begin{equation}\label{eq:step2_1}
    \P[E(y_0)]\geq  cN^{-d}\, \P[\text{CPiv}^n_{x_j}(20N)].
\end{equation}

Consider the event 
$$F(y_0,\mathcal{C}_{r})\coloneqq \{\lr{ \Lc_{20N}(x_j)}{\pazocal{I}_{n+\frac12}}{\Lc_{1}(y_0)}{\mathcal{C}_r+\Lc_{8N}}\}\cap\{\Lc_{1}(y_0)\cap (\mathcal{C}_r+\Lc_{8N})^{\mathsf{c}}\subset\pazocal{I}_{n+\frac12}\}$$
and observe that on the event $E(y_0)\cap F(y_0,\mathcal{C}_{r})$ the clusters $\mathcal{C}_r$ and $\mathcal{C}_R$ are disjoint but within $\ell^\infty$ distance at most $8N$ from each other in $\Lc_{20N}(x_j)$ -- see Figure~\ref{surgery}. In particular, one can find $x_{j'}\in (2N\Z^d)\cap \Lc_{20N}(x_j)$ which is at $\ell^\infty$ distance smaller than $6N$ from each of $\mathcal{C}_r$ and $\mathcal{C}_R$.
All in all, we obtain
\begin{equation}\label{eq:inclusion_EcapF}
    E(y_0)\cap F(y_0,\mathcal{C}_{r}) ~\subset \bigcup_{x_{j'}\in \Lc_{20N}(x_j)}~ \text{CPiv}^n_{x_{j'}}(6N).
\end{equation}
\begin{figure}[!h]
    \centering
    \begin{subfigure}[b]{0.4\textwidth}
        \centering         \includegraphics[width=0.8\textwidth]{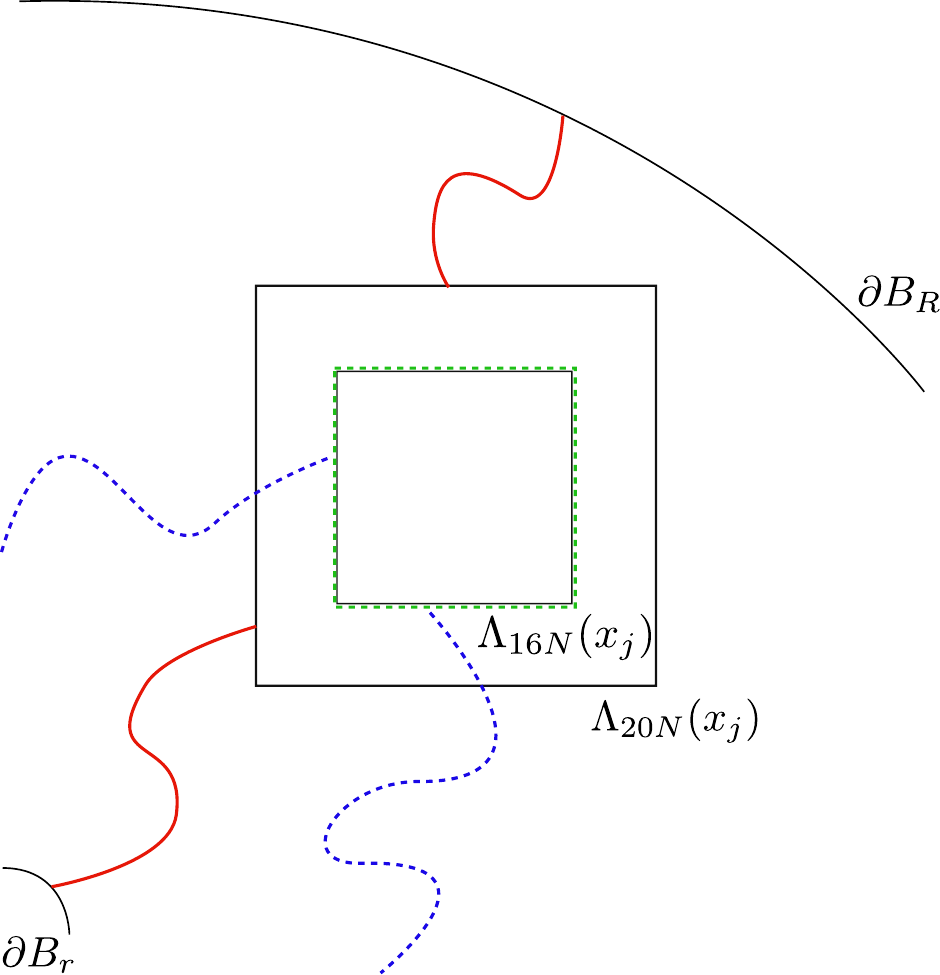}
    \end{subfigure}
    \hspace{1.5cm}
    \begin{subfigure}[b]{0.4\textwidth}
        \centering
         \includegraphics[width=0.8\textwidth]{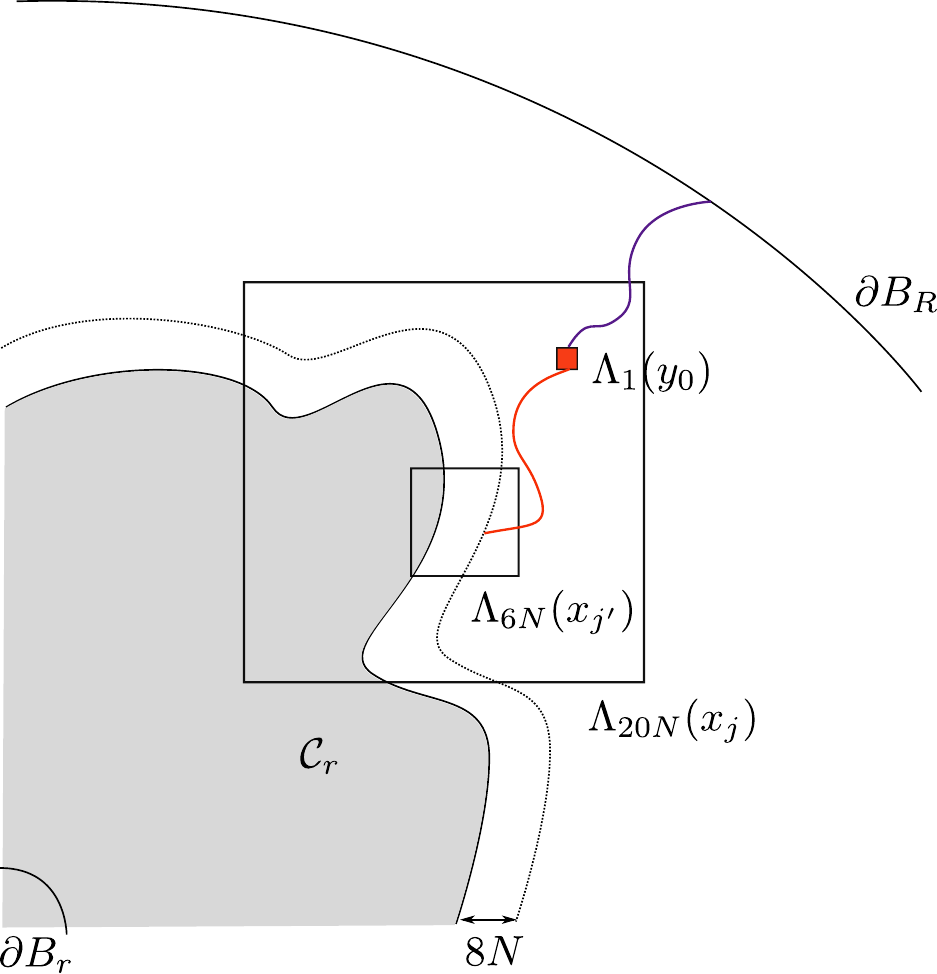}
    \end{subfigure}
    \caption{\label{surgery} On the left, an illustration of step 1. The red open paths together with the blue closed surface represent the event $\text{Piv}^n_{x_j}(20N,16N)$, while the green closed surface represents the event $F$. On the right, an illustration of step 2. The events $E(y_0)$ and $F(y_0,\mathcal C_r)$ are represented in purple and red, respectively.}
\end{figure}

For any given $\pazocal{C}\subset \R^d$, consider the event
$$F(y_0,\pazocal{C})\coloneqq \{\lr{ \Lc_{20N}(x_j)}{\pazocal{I}_{n+\frac12}}{\Lc_{1}(y_0)}{\pazocal{C}+\Lc_{8N}}\}\cap\{\Lc_{1}(y_0)\cap (\pazocal{C}+\Lc_{8N})^{\mathsf{c}}\subset\pazocal{I}_{n+\frac12}\}.$$
Similarly to \eqref{eq:bound_F}, one can prove that for every $\pazocal{C}$ such that $\pazocal{C}\cap \Lc_{20N}(x_j)\neq\emptyset$, we have
\begin{equation}\label{eq:bound_F(y_0,C)}
    \P[F(y_0,\pazocal{C})]\geq e^{-CN}.
\end{equation}
Indeed, first notice that $\P[B_1(x)\subset \pazocal{I}_{n+\frac12}]\geq c=c(\varepsilon)>0$ for all $x\in\R^d$. Then \eqref{eq:bound_F(y_0,C)} follows by covering $\Lc_{1}(y_0)$ and a straight line in $\Lc_{20N}(x_j)$ from $\Lc_{1}(y_0)$ to $\pazocal{C}+\Lc_{8N}$ with $O(N)$ many balls of radius $1$ and using the FKG inequality.
Since both events $F(y_0,\pazocal{C})$ and $E_2(y_0)$ are $(p+\tau_{n+\frac12})$-increasing in $\eta_{|(\pazocal{C}+\Lc_{4N})^{\mathsf{c}}}$, we can apply the FKG inequality to deduce that 
\begin{align}\label{eq:step2_2}
    \begin{split}
        \P[F(y_0,\pazocal{C})\cap E_2(y_0) \,\big|\, 
        \eta_{|\pazocal{C}+\Lc_{4N}}]&\geq \P[F(y_0,\pazocal{C}) \,\big|\, \eta_{|\pazocal{C}+\Lc_{4N}}] \,\, \P[E_2(y_0) \,\big|\, \eta_{|\pazocal{C}+\Lc_{4N}}]\\
        &=\P[F(y_0,\pazocal{C})] \,\, \P[E_2(y_0) \,\big|\, \eta_{|\pazocal{C}+\Lc_{4N}}]\\
        &\geq e^{-CN} \,\P[E_2(y_0) \,\big|\, \eta_{|\pazocal{C}+\Lc_{4N}}],
    \end{split}
\end{align}
for every $\pazocal{C}$ such that $\pazocal{C}\cap \Lc_{20N}(x_j)\neq\emptyset$. In the second line of \eqref{eq:step2_2} we used that $F(y_0,\pazocal{C})$ is independent of $\eta_{|\pazocal{C}+\Lc_{4N}}$ and in the third line we used \eqref{eq:bound_F(y_0,C)}. Notice that for every $\pazocal{C}$, the event $\{\mathcal{C}_{r}=\pazocal{C}\}$ is measurable with respect to $\eta_{|\pazocal{C}+\Lc_{4N}}$. 
Therefore, taking expectation of \eqref{eq:step2_2} over sets $\pazocal{C}$ such that $\pazocal{C}\cap \partial B_R=\emptyset$ and $\pazocal{C}\cap \Lc_{20N}(x_{j_0})\neq \emptyset$ under the distribution of $\mathcal{C}_{r}$, gives
\begin{equation}\label{eq:step2_3}
    \P[E(y_0)\cap F(y_0,\mathcal{C}_{r})]\geq e^{-CN}\P[E(y_0)].
\end{equation}
Combining \eqref{eq:step2_1}, \eqref{eq:inclusion_EcapF} and \eqref{eq:step2_3}, we obtain
\begin{equation}\label{eq:8Ncpiv_to_2Ncpiv}
    \P[\text{CPiv}^n_{x_j}(20N)]\leq e^{CN}\sum_{x_{j'}\in \Lc_{20N}(x_j)} \P[\text{CPiv}^n_{x_{j'}}(6N)].
\end{equation}
\vspace{0cm}

\textbf{Step 3:} \textit{From $6N$-closed-pivotal to good-$6N$-closed-pivotal or bad-$16N$-pivotal.}\\

\noindent
Fix any $x_{j'}\in \Lc_{20N}(x_j)$. Recall that $\overline{\eta}\coloneqq \overline{\eta}_o(1)$ is simply a Poisson point process of intensity $1$ on $\R^d$. Consider the following events 
\begin{align}
    \pazocal{E}&\coloneqq \left\{ |\overline{\eta}\cap \Lu_N(y)|\leq 2N^d ~~~ \forall y\in N\Z^d \text{ with } \Lu_N(y)\subset \Lc_{12N}(x_{j'})\right\}, \\
    \pazocal{F}&\coloneqq \left\{ \overline{\eta}\cap B_N(y)\neq \emptyset ~~~\forall y\in \Lc_{12N}(x_{j'}) \right\},
\end{align}
and finally
\begin{equation*}
    \pazocal{G} \coloneqq (x_{j'} + \pazocal{D}(6N,N^\gamma) ) \cap \pazocal{E} \cap \pazocal{F}
\end{equation*}
where we recall that $\pazocal D(\cdot,\cdot)$ was defined in the statement of Proposition~\ref{prop:chemical_dist}.
On the one hand, a straightforward computation with Poisson random variables gives $\P[\pazocal{E} \cap \pazocal{F}]\geq 1-e^{-cN^d}$. On the other hand, Proposition~\ref{prop:chemical_dist} implies that $\P[\pazocal{D}(6N,N^\gamma)]\geq 1-e^{-cN^\gamma}$ (recall that $\gamma=d-\delta\in(d-1,d)$), and therefore
\begin{equation*}
    \P[\pazocal{G}]\geq 1-e^{-cN^\gamma}.
\end{equation*}
Notice that $\text{Piv}^n_{x_{j'}}(16N)$ is $\eta_{|\Lc_{14N}(x_{j'})^{\mathrm{c}}}$-measurable and $\pazocal{G}$ is $\eta_{|\Lambda_{13N}(x_{j'})}$-measurable, so these events are independent. As a conclusion, we have
\begin{align*}
    \begin{split}
        \P[\text{CPiv}^n_{x_{j'}}(2\overline{N})\cap \pazocal G^{\mathsf{c}}]&\leq \P[\text{Piv}^n_{x_{j'}}(6\overline{N})\cap \pazocal{G}^{\mathsf{c}}]\\
        &\leq e^{-cN^{\gamma}}\P[\text{Piv}^n_{x_{j'}}(6\overline{N})]=e^{-cN^\gamma}p_n(x_{j'}),
    \end{split}
\end{align*}
and hence
\begin{equation}\label{eq:piv_to_goodclosepiv}
    \P[\text{CPiv}^n_{x_{j'}}(2\overline{N})] \leq \P[\text{CPiv}^n_{x_{j'}}(2\overline{N})\cap\pazocal G] + e^{-c N^\gamma} p_n(x_{j'}).
\end{equation}
\vspace{0cm}

\textbf{Step 4:} \textit{From good-$6N$-closed-pivotal to sprinkling-pivotal.}\\ 

\noindent

First observe that on the event $\pazocal{E}$, the configuration $\omega_a^{n+\frac12}$ coincides with $\overline{\eta}_a(p+\tau_{n+\frac12})$ in $\Lc_{12N}(x_{j'})$, for $a\in\{o,c\}$, and on the event $\pazocal{F}$ every point in $\Lc_{12N}(x_{j'})$ remains at distance at most $N$ from $\overline{\eta}$. Second, notice that conditioning on $\overline{\eta}_o(p+\tau_{n+\frac12})$ and $\overline{\eta}_c(p+\tau_{n+\frac12})$ completely determines $\pazocal{I}_{n+\frac12}$ and $\overline{\eta}$, and therefore the clusters $\mathcal{C}_r$ and $\mathcal{C}_R$ along with the event $\text{CPiv}^n_{x_{j'}}(6N)\cap \pazocal{G}$. Third, on the event $\text{CPiv}^n_{x_{j'}}(6N)\cap \pazocal{G}$, there exist a path $w_1,w_2,\cdots,w_t \in G(\overline{\eta})\cap \Lc_{12N}(x_{j'})$, with $t\leq N^\gamma$, between the cells of $y\in \mathcal{C}_r\cap \Lc_{6N}(x_{j'})$ and $z\in \mathcal{C}_R\cap \Lc_{6N}(x_{j'})$. It follows that if $w_s\in \overline{\eta}_o(p+\tau_{n+1})$ for all $1\leq s\leq t$, then $y$ is connected to $z$ in $\pazocal{I}_{n+1}$, thus implying that $\{\pazocal{I}_{n+1}\in A\}$ happens. Since conditionally on $\overline{\eta}_o(p+\tau_{n+\frac12})$ and $\overline{\eta}_c(p+\tau_{n+\frac12})$, each point $w_s\in\overline{\eta}$ has probability at least $(\tau_{n+1}-\tau_{n+\frac12})(w_s)\asymp (\tau_{n+1}-\tau_{n+\frac12})(x_j)=e^{-N^{1-\delta}}\tau(\frac{x_j-x_n}{2N})$ of belonging to $\overline{\eta}_o(p+\tau_{n+1})$, it follows that the the conditional probability of $\{\pazocal{I}_{n+1}\in A\}$ given $\text{CPiv}^n_{x_{j'}}(6N)\cap \pazocal{G}$ is at least $\prod_{s=1}^t(\tau_{n+1}-\tau_{n+\frac12})(w_s)\geq e^{-C N^{\gamma'}} \tau\big(\tfrac{x_{j}-x_n}{2N}\big)^{N^\gamma}$  (recall \eqref{eq:choice_epsilon} and \eqref{eq:tau_k_2}). 
All in all, we obtain
\begin{equation}\label{eq:goodclosepiv_to_b}
\begin{split}
    \P[\text{CPiv}^n_{x_{j'}}(6N)\cap \pazocal{G}]&\leq e^{C N^{\gamma'}}\tau\big(\tfrac{x_{j}-x_n}{2N}\big)^{-N^\gamma}\, \P[\pazocal{I}_{n+1}\in A,~\pazocal{I}_{n+\frac12}\notin A]\\
    &= e^{C N^{\gamma'}}\tau\big(\tfrac{x_{j}-x_n}{2N}\big)^{-N^\gamma}\, q_n
\end{split}
\end{equation}
Combining \eqref{eq:4Npiv_to_8Ncpiv}, \eqref{eq:8Ncpiv_to_2Ncpiv}, \eqref{eq:piv_to_goodclosepiv}, \eqref{eq:goodclosepiv_to_b} and reminding that $\gamma>d-1$, one readily obtains the desired bound \eqref{eq:conv_bound}, thus concluding the proof.
\end{proof}

\section{Local uniqueness}\label{sec:loc_uniq}

In this section we will prove that local uniqueness follows from the fast decay of disconnection obtained in Corollary~\ref{cor:decay_disconnection}. The proofs of this section are adaptation of existing proofs and can be generalized to other models with fast decay of disconnection (as in Corollary~\ref{cor:decay_disconnection}) as long as the model satisfies a ``sprinkled finite energy'' property, meaning that there is a positive (conditional) probability to open a region after a sprinkling.

Notice that for $d\geq3$, the exponent $d-1$ in Corollary~\ref{cor:decay_disconnection} is strictly greater than $1$, which will be helpful in our proof.

Our first step is to prove that, for every $p>p_c$, there exists a \emph{dense cluster} of $\pazocal{V}(p)$ in $\Lc_L$.  We say that a (connected) set $C\subset \R^d$ is \emph{$\ell$-dense} in $\Lc_L$ if $C\cap \Lc_\ell(x)\neq \emptyset$ for every $x\in \ell\Z^d\cap\Lc_L$. Let $\mathcal{D}(L,\ell, p)$ be the event that there exists a cluster of $\pazocal{V}(p)\cap \Lc_{2L}$ which is $\ell$-dense in $\Lc_L$. The proof of the following proposition is inspired by \cite{benjaminitassion17} (see also \cite{DGRS19,DGRST23b,Sev24} for applications of this technique to other models).

\begin{proposition}[Dense cluster]\label{prop:dense_cluster}
   For every $d\geq3$ and $p>p_c$, there exists a constant $C_u=C_u(p)<\infty$ such that 
    \begin{equation*}
        \P[\mathcal{D}(L,C_u(\log L)^{\frac{1}{d-1}},p)]\to 1 ~~~~ \text{as } L\to\infty.
    \end{equation*}
\end{proposition}

\begin{proof}
    Fix $p>p'>p_c$ and consider the standard coupling between $\pazocal{V}(p')$ and $\pazocal{V}(p)$ where both have the exact same tessellation, as defined in Section~\ref{sec:comparison}.
    We also fix $L\geq1$, which we may henceforth drop from the notation. 
    Let $C_u=d/c(p')$, where $c(p')$ is the constant given by Corollary~\ref{cor:decay_disconnection}, and set $\ell=\ell(L,p')=C_u(\log L)^{\frac{1}{d-1}}$. By Corollary~\ref{cor:decay_disconnection}, one has $\P[\lr{}{\pazocal{V}(p')}{\Lc_\ell(x)}{\infty}]\geq 1-L^{-d}$. 
    By union bound, one concludes that the event 
    $$\pazocal{A}=\pazocal{A}_L\coloneqq \bigcap_{x\in \ell\Z^d\cap\Lc_{2L}} \{\lr{}{\pazocal{V}(p')}{\Lc_\ell(x)}{\partial \Lc_{2L}}\}$$
    satisfies $\P[\pazocal{A}_L]\to1$ as $L\to\infty$. We also consider the ``good event'' 
    $$\pazocal{G}=\pazocal{G}_L\coloneqq \pazocal G_{2L,\ell,\kappa}\cap \big\{ |\{x\in\overline \eta : C(x)\cap \Lambda_{2L}\neq \emptyset\}| \leq CL^d \big\}.$$ 
    By Lemmas~\ref{lemma:intersection boundary} and \ref{lemma :GL}, we can choose constants $C,\kappa$ such that $\P[\pazocal{G}_L]\to1$ as $L\to\infty$. Also notice that $\pazocal{G}$ is $\overline{\eta}$-measurable. Our goal now is to prove that $\mathcal{D}(L,C_u(\log L)^{\frac{1}{d-1}},p)$ happens with high probability conditionally on $\pazocal{A}\cap\pazocal{G}$.
    Consider the set of boundary clusters
    \begin{equation*} 
    \mathcal{C}\coloneqq \{C\subset \Lc_{2L}:~ \text{$C$ is a cluster in $\pazocal{V}(p')\cap\Lc_{2L}$ such that $C\cap\partial \Lc_{2L}\neq\emptyset$}\}.
    \end{equation*} 
    Given a subset $\omega\subset\R^d$ such that $\pazocal{V}(p')\subset \omega$, we define the relation $\sim_\omega$ in $\mathcal{C}$ by setting
    $$C\sim_\omega C' \text{ if } \lr{\Lc_{2L}}{\omega}{C}{C'}.$$
    It is enough to prove that, with high probability conditionally on $\pazocal{A}\cap\pazocal{G}$, all the clusters in $\mathcal{C'}\coloneqq\{C\in\mathcal{C}:~ C\cap\Lc_L\neq\emptyset\}$ are connected to each other in $\pazocal{V}(p)\cap\Lc_{2L}$, or equivalently $|\mathcal{C}'/\sim_{\pazocal{V}(p)}|=1$
    
    Let $V_i\coloneqq \Lc_{2L-i\sqrt{L}}$, $0\leq i\leq \lfloor\sqrt{L}\rfloor$. We will interpolate between the models $\pazocal{V}(p')$ and $\pazocal{V}(p)$ in $\Lambda_{2L}$ by progressively sprinkling in each annulus $V_{i}\setminus V_{i+1}$ as follows. For every $0\leq i\leq \lfloor\sqrt{L}\rfloor$, we define 
    $$\overline{\eta}_o^i\coloneqq (\overline{\eta}_o(p)\cap (V_0\setminus V_{i}))\cup (\overline{\eta}_o(p')\cap (V_0\setminus V_{i})^{\mathrm{c}}),$$ 
    $$\pazocal{V}^i\coloneqq \{y\in\R^d:~ d(y,\overline{\eta}_o^i)\leq d(y,\overline{\eta})\}.$$ 
    Given $\omega\supset \pazocal{V}(p')$, let
    $$\pazocal{U}_i(\omega)\coloneqq \mathcal{C}_i/\sim_\omega \,,$$
    where 
    $$\mathcal{C}_i\coloneqq \{C\in\mathcal{C}:~ C\cap V_i\neq\emptyset\}.$$
    Finally, we set 
    $$U_i\coloneqq |\pazocal{U}_i(\pazocal{V}^i)|.$$
    Notice that it is enough to prove that $U_{\lfloor\sqrt{L}\rfloor}=1$ with high probability conditionally on $\pazocal{A}\cap \pazocal{G}$, which follows from the following lemma.

    \begin{lemma}\label{lem:gluing}
        For every $0\leq i\leq \lfloor\sqrt{L}\rfloor-8$, one has
        \begin{equation}\label{eq:gluing}
            \P[U_{i+8}> 1\vee U_i/2 \,\big|\, \pazocal{A}\cap\pazocal{G}]\leq e^{-cL^{1/4}}.
        \end{equation}
    \end{lemma}
    \noindent  By Lemma~\ref{lem:gluing} together with a union bound, the following event occurs with high probability
    \[\bigcap_{0\leq i\leq \lfloor\sqrt{L}\rfloor-8} \{U_{i+8}\le  1\vee U_i/2\} \cap \pazocal{A}\cap\pazocal{G}.\]
    On this event, $U_{\lfloor\sqrt{L}\rfloor}>1$ would imply $U_0\ge 2^ {\lfloor\sqrt{L}\rfloor/8-1}$, which contradicts the fact that $U_0\leq CL^{d}$ on $\pazocal{G}$. This yields that $U_{\lfloor\sqrt{L}\rfloor}=1$ on this event, thus concluding the proof. 
\end{proof}

It remains to give the following proof.
\begin{proof}[Proof of Lemma~\ref{lem:gluing}]
    Fix $0\leq i\leq \lfloor\sqrt{L}\rfloor-8$. Roughly speaking, we want to prove that, with very high probability, each cluster hitting $V_i$ either gets merged with another cluster after sprinkling in the annulus $V_i\setminus V_{i+8}$ or dies out before hitting $V_{i+8}$.
    We start by finding a sub-annulus of $V_i\setminus V_{i+8}$ where only a small proportion of clusters die out. For every $\omega\supset \pazocal{V}(p')$ and $j\in\{0,1,2,3\}$, let
    $$\pazocal{U}_i^j(\omega)\coloneqq \{C\in\pazocal{U}_i(\omega):~ C\cap V_{i+2j}\neq\emptyset \text{ and } C\cap V_{i+2j+2}=\emptyset\}.$$
    In the definition above, we abuse the notation by identifying the equivalence class of cluster $C\in\pazocal{U}_i(\omega)$ with its associated $\omega$-cluster. Since $(\pazocal{U}_i^j(\omega))_j$ are disjoint subsets of $\pazocal{U}_i(\omega)$, we can find $j\in\{0,1,2,3\}$ such that $|\pazocal{U}_i^j(\pazocal{V}_i)|\leq |\pazocal{U}_i(\pazocal{V}_i)|/4=U_i/4$. We fix such a $j$ for the rest of the proof and focus on the annulus $V_{i+2j}\setminus V_{i+2j+2}$. 
    
    We now further restrict to one of the two sub-annuli $V_{i+2j}\setminus V_{i+2j+1}$ and $V_{i+2j+1}\setminus V_{i+2j+2}$ as follows. If at least one of the clusters in $\pazocal{U}_i^j(\pazocal{V}_i)$ touches $V_{i+j+1}$, we ``wire'' all clusters in $\pazocal{U}_i^j(\pazocal{V}_i)$ (i.e.~we treat their union as a single element) and focus on the annulus $V_{i+2j}\setminus V_{i+2j+1}$. Otherwise, we forget about all the clusters in $\pazocal{U}_i^j(\pazocal{V}_i)$ and focus on the annulus $V_{i+2j+1}\setminus V_{i+2j+2}$. See Figure~\ref{partition} for an illustration of both cases.
    More precisely, consider the family
    \begin{equation*}
        \tilde{\pazocal{U}}\coloneqq \left\{ \begin{array}{ll}
    \{C\in\pazocal{U}_i(\pazocal{V}_i):~ C\cap V_{i+2j+2}\neq\emptyset\}\cup \big\{\tilde{C}\big\},     &  ~~\text{if } \tilde{C}\cap V_{i+2j+1}\neq \emptyset, \\
    \{C\in\pazocal{U}_i(\pazocal{V}_i):~ C\cap V_{i+2j+2}\neq\emptyset\},     & ~~\text{otherwise.}
    \end{array}
    \right.
    \end{equation*}
    where $\tilde{C}\coloneqq \cup_{C\in\pazocal{U}_i^j(\pazocal{V}_i)} C$ if $\pazocal{U}_i^j(\pazocal{V}_i)\neq \emptyset$ and $\tilde{C}=\emptyset$ otherwise, and the annulus
    \begin{equation*}
        A\coloneqq \left\{ \begin{array}{ll} V_{i+2j}\setminus V_{i+2j+1},  & ~~\text{if } \tilde{C}\cap V_{i+2j+1}\neq \emptyset, \\
    V_{i+2j+1}\setminus V_{i+2j+2}    & ~~\text{otherwise.}\end{array}\right.
    \end{equation*}
\begin{figure}[!h]
    \centering
    \begin{subfigure}[b]{0.4\textwidth}
        \centering         \includegraphics[width=0.9\textwidth]{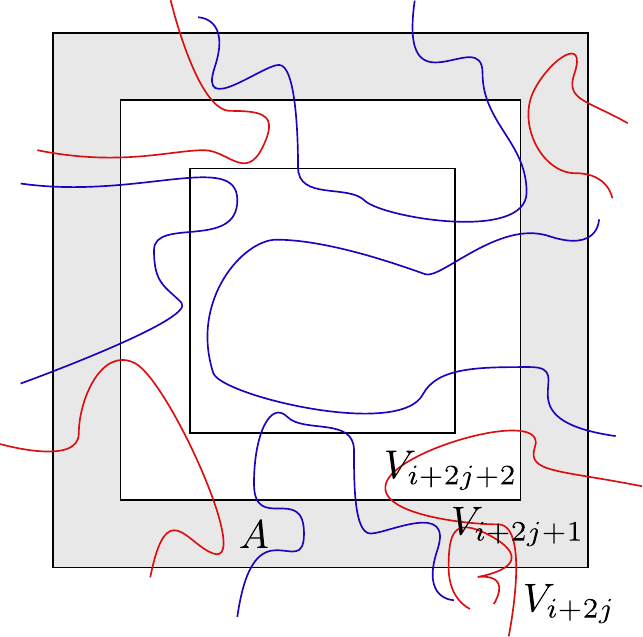}
    \end{subfigure}
    \hspace{1.5cm}
    \begin{subfigure}[b]{0.4\textwidth}
        \centering
         \includegraphics[width=0.9\textwidth]{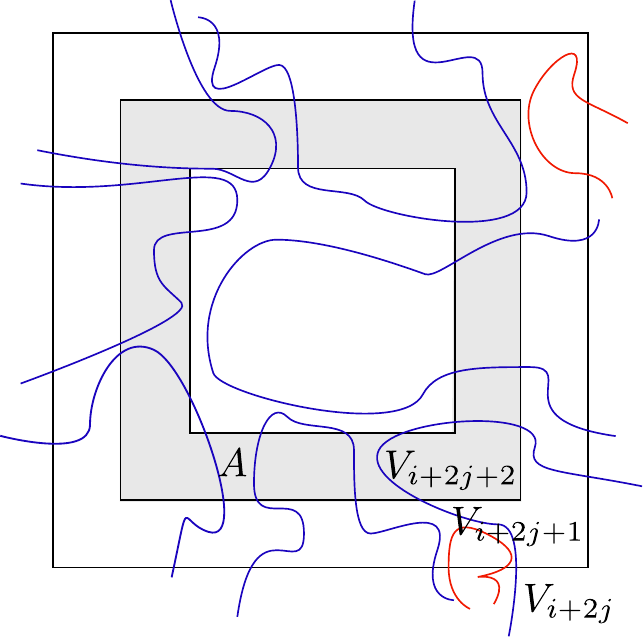}
    \end{subfigure}
    \caption{\label{partition} The clusters in $\pazocal{U}_i^j(\pazocal{V}_i)$, whose union is $\tilde{C}$, are represented in red, while clusters in $\pazocal{U}_i(\pazocal{V}_i)$ that intersect $V_{i+2j+2}$ are represented in blue. On the left we have $\tilde{C}\cap V_{i+2j+1}\neq \emptyset$ and $\tilde{\pazocal{U}}$ consists of blue clusters and $\tilde{C}$ seen as a single cluster. On the right we have $\tilde{C}\cap V_{i+2j+1}=\emptyset$ and $\tilde{\pazocal{U}}$ consists of blue clusters only. In the annulus $A$, the family $\tilde{\pazocal{U}}$ is $\ell$-dense and each of its elements is crossing.}
\end{figure}

Notice that in any case, every element of  $\tilde{\pazocal{U}}$ contains a crossing of the annulus $A$, and that every box $\Lc_\ell(x)\subset A$ intersects at least one element of $\tilde{\pazocal{U}}$. 
Define the following partition of $\tilde{ \pazocal U}$. Let $\overline {\pazocal U}_0$ be the set made of all the elements of $\tilde{ \pazocal U}$ connected to $\tilde C$ in $\pazocal{V}_{i+8}$ and $\overline{ \pazocal U}_1,\dots,\overline {\pazocal U}_m$ be the partition of $\tilde {\pazocal U}\setminus \overline{ \pazocal U}_0$ induced by the equivalence relation $\sim_{\pazocal{V}_{i+8}}$. 
If we assume that $|\overline {\pazocal U}_k|\ge 4$ for all $1\le k \le m $, then we get
    \[U_{i+8}\le\frac 1 4 |\tilde {\pazocal U}\setminus \overline{ \pazocal U}_0|+|\pazocal{U}_i^j(\pazocal{V}_i)|\le \frac 1 2U_i.\]
This cannot occur on the event $\pazocal{E}\coloneqq \{U_{i+8}> 1\vee U_i/2\} \cap \pazocal{A}\cap\pazocal{G}$.
As a result, if $\pazocal{E}$ happens, then there exists a non-trivial partition $\tilde{\pazocal{U}}=\tilde{\pazocal{U}}_1\cup \tilde{\pazocal{U}}_2$ such that $|\tilde{\pazocal{U}}_1|< 4$ and $\nlr{A}{\pazocal{V}_{i+8}}{\tilde{\pazocal{U}}_1}{\tilde{\pazocal{U}}_2}$, where here we abuse the notation again by identifying $\tilde{\pazocal{U}}_1$ and $\tilde{\pazocal{U}}_2$ with the union of the associated $\pazocal{V}_i$-clusters.
As a conclusion, we have
\begin{equation}\label{eq:gluing_proof1}
    \P[\pazocal{E}]\leq \E \Big[ \,1_{\pazocal{G}}\, \sum_{\tilde{U}} \P[\{\tilde{\pazocal{U}}=\tilde{U}\}\cap\pazocal{A} \,|\, \overline{\eta}\, ] \, \sum_{\substack{\tilde{U}=\tilde{U}_1\sqcup \tilde{U}_2 \\ 1\leq|\tilde U_1|< 4}} \P[\,\nlr{A}{\pazocal{V}_{i+8}}{\tilde U_1}{\tilde U_2} \,|\, \{\tilde{\pazocal{U}}=\tilde{U}\}\cap\pazocal{A},\, \overline{\eta}\, ] \Big].
\end{equation}
Notice that the sum above makes sense since there are only finitely many possibilities for $\tilde{U}$ given $\overline{\eta}$ (which determines the graph $G(\overline{\eta})$). Also recall that $\pazocal{G}$ is $\overline{\eta}$-measurable.
    
For any pair of set $\tilde{U}_1$, $\tilde{U}_2$ as in \eqref{eq:gluing_proof1}, one can apply Lemma~\ref{lem:crossing_boundary} below to the sets $\{x\in \Z^d: \Lambda_\ell(\ell x)\cap \tilde U_1\neq\emptyset\}$ and  $\{x\in \Z^d: \Lambda_\ell(\ell x)\cap \tilde U_2\neq \emptyset\}$ to obtain a $*$-path $x_1,\cdots,x_n\in  \Z^d$, $n\geq c_d\sqrt{L}/\ell$, crossing $A/\ell$ such that $\Lc_{3\ell}(\ell x_k)$ intersects both $\tilde{U}_1$ and $\tilde{U}_2$.
On the event $\pazocal G_{2L,\ell,\kappa}\subset \pazocal{G}$, we have $|\{i\in\{1,\dots,n\}: \ell x_i+\pazocal{D}(3\ell,\kappa\ell)\text{ occurs}\}|\ge n/2$.
Hence, there exist $y_1,\cdots,y_{m}\subset \{x_1,\cdots,x_n\}$ with $m\ge \frac{n}{2(12)^d}$, such that the events  $\ell y_i+\pazocal{D}(3\ell,\kappa\ell)$ occur and $\|y_i-y_j\|_{\infty}\geq 12$ for all $i\neq j$. In particular, there exists a path in $G(\overline\eta)\cap \Lambda_{6\ell}(\ell y_k)$ of length less than $\kappa \ell$ joining $\tilde U_1$ and $\tilde U_2$. Therefore, on the event $\pazocal{G}$, we have
\begin{equation}\label{eq:gluing_proof2}
    \begin{split}
        \P[\, \nlr{A}{\pazocal{V}_{i+8}}{\tilde{U}_1}{\tilde{U}_2} \,|\, \{\tilde{\pazocal{U}}=\tilde{U}\}\cap\pazocal{A},\, \overline{\eta}\,] &\leq \P\Big[\bigcap_{k=1}^{n} \nlr{D_k}{\pazocal{V}_{i+8}}{\tilde{U}_1}{\tilde{U}_2} \,|\, \{\tilde{\pazocal{U}}=\tilde{U}\}\cap\pazocal{A},\, \overline{\eta}\,\Big] \\
        &\leq \big(1-(p-p')^{\kappa \ell}\big)^m
        \leq e^{-cL^{1/4}},
        \end{split}
\end{equation}
for some constant $c>0$, where $D_k\coloneqq \Lambda_{6\ell}(\ell y_k)$ are disjoint domains. In the last line of \eqref{eq:gluing_proof2} we used the fact that each cell in this path is open in $\pazocal{V}_{i+8}$ with conditional probability at least $p-p'$. Combining \eqref{eq:gluing_proof1}, \eqref{eq:gluing_proof2} and the fact that $|\tilde U|\leq CL^d$ on $\pazocal{G}$, we obtain
$$\P[\pazocal{E}]\leq \sum_{k=1}^3 {CL^d\choose k} e^{-cL^{1/4}}\leq e^{-c'L^{1/4}},$$
for some $c'>0$ and $L$ large enough, thus concluding the proof. 
\end{proof}

We have used the following deterministic lemma, which can be found in \cite[Lemma 2.1]{DGRST23b}.
\begin{lemma}\label{lem:crossing_boundary}
Let $U, V \subset \Z^d$ be such that $U \cup V = (\Lc_m \setminus \Lc_{n-1})\cap \Z^d$ and both contain a crossing from $\partial \Lc_n$ to $ \partial \Lc_m$. There exists a $*$-path $\pi \equiv (\pi(i))_{1 \le i \le |\pi|}$ joining $\partial \Lc_n$ and $ \partial \Lc_m$ with
    \begin{equation*}\label{eq:cond_path}
    d_\infty(\pi(i), U) \vee d_\infty(\pi(i), V)  \le 1 \text{\,\,for all $1 \le i \le |\pi|$.}
    \end{equation*}
\end{lemma}

With the existence of a dense cluster in hands, we can now prove uniqueness of crossing clusters.

\begin{proof}[Proof of Theorem~\ref{thm:unique}]
Fix $p>p'>p_c$ and consider the standard coupling between $\pazocal{V}(p')$ and $\pazocal{V}(p)$ introduced in Section~\ref{sec:comparison} and recall that $G(\overline \eta)$ denotes the embedded graph induced by the Voronoi tessellation of $\overline{\eta}$, see Section~\ref{sec:preliminaries}.
Condition on the event 
\[\pazocal A_L\coloneqq\mathcal{D}(L,\ell,p')\cap \cG_{L,\ell,\kappa}\cap \{\forall  x\in \overline \eta : C(x)\cap \Lambda_L\ne \emptyset, \mathrm {diam}(C(x))\le \ell\}\cap \{|\{x\in\overline \eta : C(x)\cap \Lambda_L\}| \leq CL^ {d}\},\] 
where $\ell=C_u(p')(\log L)^{\frac{1}{d-1}}$. By Proposition~\ref{prop:dense_cluster}, Lemmas \ref{lemma:diam},  \ref{lemma:intersection boundary} and \ref{lemma :GL}, we get that $\P[\pazocal A_L]\to1$ as $L\to\infty$.
Now condition on $\overline \eta$ and the $p'$-open cells. Note that this information is sufficient to determine $\pazocal A_L$ and that we work conditionally on $\pazocal A_L$. 
For all $x\in \ell \Z^ d$ such that $x+\pazocal{D}(3\ell,\kappa\ell )$ occurs, denote by $\mathcal B(x)$ the set of cells intersecting $\Lambda_{6\ell}(x)$ of chemical distance less than $\kappa \ell$ from the cell $C(x)$, that is
\[\mathcal B(x)\coloneqq \{y\in \overline \eta :\, C(y)\cap \Lambda_{6\ell}(x)\ne\emptyset ,\,  \text{there exists a path between $x$ and $y$ in $G(\overline \eta)$ of length at most $\kappa \ell$}\}. \]
By definition of the event  $x+\pazocal{D}(3 \ell,\kappa\ell )$, we get $\Lambda_{3\ell}(x)\subset\cup_{y\in\mathcal B(x)}C(y)$. Moreover, on $\pazocal A_L$, we have $\cup_{y\in\mathcal B(x)}C(y)\subset \Lambda_{7\ell}(x)$.

Let $w\in\overline \eta$ such that $C(w)\cap \partial \Lambda_{L/2}\ne \emptyset$. Recall that we condition on the tiling $G(\overline{\eta})$ and the $p'$-state of each cell, which on the event $\pazocal{A}_L$ implies the existence of a $p'$-open cluster $\mathcal{C}$ which is $\ell$-dense in $\Lc_L$. Our goal is to deduce that with very high probability (actually, stretched exponential close to one as $L$ grows) either the cluster of $w$ does not intersect $\partial\Lc_{L}$ or it is connected to $\mathcal{C}$.
We will prove this by exploring the cluster of $w$ in $\pazocal{V}(p)$, revealing one by one the state of the cells that are neighbours to $p$-open cells of the cluster of $w$ that have been explored. At each stage of our exploration, we stop when one of these following cases occur:
\begin{itemize}
    \item there is no more cells to explore (all the neighbours of the open cells have been explored), or all the cells left to explore have their center outside $\Lambda_{2L}$,
    \item after revealing the state of an explored cell, we discover that the cluster of $w$ is connected to $\mathcal{C}$.
\end{itemize}
If at some stage of the exploration we examine the first cell $C(y)$, $y\in\overline\eta $, of $\mathcal B(x)$ for some $x\in \ell \Z^ d$ such that $x+\pazocal D(3\ell,\kappa\ell )$ occurs, then, we immediately examine all the cells in $\mathcal B(x)$ at the same time. Recall that the dense cluster $\mathcal{C}$ at $p'$ intersects $\Lambda_\ell(x)$ and so it contains the cell of some $z\in \mathcal B(x)$. Moreover, by definition, the graph distance between $z$ and $y$ within $\mathcal B(x)$ is less than $2 \kappa \ell$. By construction of the exploration, the cells in $\mathcal B(x)$ have not been explored. It follows that there is a probability at least $(p-p')^ {2\kappa\ell}$ that the $p$-cluster of $w$ connects with $\mathcal{C}$ in $\mathcal B(x)$. If this happens, we stop the exploration. Otherwise, we reveal all the cells in $\mathcal B(x)$ and resume the exploration. 

Let us assume that the cluster of $w$ intersects $\partial \Lambda_L$ but does not intersect $\mathcal{C}$. In particular, the exploration stops when all the cells left to explore have their centers outside $\Lambda_{2L}$. Hence, the explored region has diameter at least $L/2$. On the event $\cG_{L,\ell,\kappa}$, the exploration has encountered at least $L/{2\ell}$ sets of the form $\mathcal B(x)$ for some $x\in \ell \Z^ d$ such that $x+\pazocal{D}(3 \ell,\kappa\ell )$ occurs. In particular, there exist $x_1,\dots,x_m$ such that $m\ge \frac L{2\ell (14) ^ d} $ such that $x_i+\pazocal{D}(3\ell,\kappa\ell )$ occurs for all $i$ and $\mathcal B(x_i)\cap \mathcal B(x_j)=\emptyset$ for $i\ne j$.
Finally, it yields that conditionally on $\pazocal A_L$, the probability that the $p$-cluster of $w$ crosses $\Lambda_L\setminus \Lambda_{L/2}$ without intersecting the dense cluster is at most
\begin{equation*}
    \big( 1- (p-p')^{2\kappa\ell} \big)^{ \frac{L}{2\ell (14)^d} } \leq e^{-cL^{1/2}},
\end{equation*}
where we have used that $\ell =o(\log L)$ for $d\ge 3$.
We conclude the proof by a union bound over all the cells intersecting $\partial \Lambda_{L/2}$ (on the event $\pazocal A_L$ there are at most $CL^d$ such cells).
\end{proof}

The proof of Theorem~\ref{thm:sharpness} will follow from a standard renormalization argument. We need to define a notion of good and bad box. 
We say that a site $x\in \Z^ d$ is $N$-good if there exists a unique cluster of diameter larger than $N/4$ in $\Lambda_{4N}(2Nx)$ and this cluster moreover intersects all the $3^d$ sub-boxes of the form $\Lambda_N(2Ny)$ with $y$ such that $\|y-x\|_\infty \le 1$. We say that $x$ is $N$-bad if it is not $N$-good. The following lemma states that a box is typically good for large $N$.
\begin{lemma}\label{lem:contgoodbox}For every $d\ge 3$ and $p>p_c(d)$ and $x\in\Z^d$, one has
    \begin{equation}\label{eq:good}
    \P[\text{$x$ is $N$-good}] \rightarrow 1 ~\text{ as } N\to\infty.
    \end{equation}
    
\end{lemma}
\begin{proof}[Proof of Lemma \ref{lem:contgoodbox}]
Let $L=N/4$. Without loss of generality let us prove the result for $x=0$. 
Define by $\overline{\pazocal{U}}(L,p)$ the event that $\pazocal{V}(p)\cap \Lc_{2L}$ contains only one cluster crossing the annulus $\Lc_L\setminus \Lc_{L/2}$, which in turn crosses the annulus $\Lc_{2L}\setminus \Lc_{L/4}$ as well.
In other words, we have $\overline{\pazocal{U}}(L,p)=\pazocal{U}(L,p)\cap\{\lr{}{\pazocal{V}(p)}{\Lc_{L/4}}{\partial \Lc_{2L}}\}$, so combining Theorem~\ref{thm:unique} and Corollary~\ref{cor:decay_disconnection}, we obtain
    \begin{equation}\label{eq:uniquedef2}
    \P[\overline{\pazocal{U}}(L,p)] \rightarrow 1 ~\text{ as } L\to\infty.
    \end{equation}
Let us denote by $\cE$ the following event
\[\cE\coloneqq\bigcap _{x\in\frac L 4\Z^ d\cap  \Lambda_{3N}}(x+\overline{\pazocal{U}}(L,p)).\]
By union bound, we get
\begin{equation*}
    \P[\cE^{\mathrm{c}}]\le 2\left(\frac {24N}L\right) ^ d\P[\overline{\pazocal{U}}(L,p)^{\mathrm{c}}].
\end{equation*}
In particular, thanks to \eqref{eq:uniquedef2}, we have that the probability of $\cE$ goes to $1$ with $N$.

We claim that on the event $\cE$ the box $0$ is $N$-good. Let us assume that there exist two distinct clusters $C_1,C_2$ of diameter at least $L$ in $\Lambda_{3N}$. Then there exist $x,y \in\frac L 4\Z^ d\cap  \Lambda_{3N}$ such that $C_1$ crosses $x+\Lambda_L\setminus \Lambda_{L/2}$ and $C_2$ crosses $x+\Lambda_L\setminus \Lambda_{L/2}$. Let us first prove that if $x$ and $y$ are neighbours then the two clusters are connected in $\Lambda_{3N+2L}$. By definition of $x+\overline{\pazocal{U}}(L,p)$, $C_1$ also crosses the annulus $x+ \Lambda_{2L}\setminus \Lambda_{L/4}$. Since $x+\Lambda_{L/4}\subset y+\Lambda_{L/2}$ and $y+\Lambda_L\subset x+\Lambda_{2L}$, then $C_1$ also crosses $y+\Lambda_L\setminus \Lambda_{L/2}$. On the event $y+\overline{\pazocal{U}}(L,p)$, the two clusters $C_1$ and $C_2$ are connected in $\Lambda_{3N+2L}$. In the case where $x$ and $y$ are not neighbours, we fix any deterministic path between $x$ and $y$ inside $\frac L4\Z^ d\cap  \Lambda_{3N}$ and by iterating the previous argument, we prove that on the event $\cE$, the clusters $C_1$ and $C_2$ are connected in $\Lambda_{3N+2L}$. The claim follows.
    
\end{proof}
\begin{proof}[Proof of Theorem~\ref{thm:sharpness}]
 Let $\delta>0$ be small enough to be chosen later depending on $d$. The event $\{\text{$x$ is $N$-good}\}$ is not short-range dependent due to the potential presence of arbitrary large cells intersecting $\Lambda_{3N}(2Nx)$. For that reason, we define the notion of being $N$-super-good as the event of being $N$-good and that all the cells intersecting $\Lambda_{3N}(2Nx)$ are contained in $\Lambda_{4N}(2Nx)$. We say that $x$ is $N$-super-bad if it is not $N$-super-good. Since the maximum diameter of a cell in $\Lambda_{4N}(2Nx)$ is $8\sqrt d N$, then the event that $x$ is $N$-super-good only depends on the configuration inside $\Lambda_{4(1+2\sqrt d)N}(2Nx)$. 
 Note that if there is a cell intersecting $\Lambda_{3N}(2Nx)$ and not contained in $\Lambda_{4N}(2Nx)$, then there is a cell of diameter at least $N$ intersecting $\Lambda_{3N}(2Nx)$. Note that by Lemma~\ref{lemma:diam}, we have
that there exists $c>0$ such that 
\begin{equation}
    \P[\text{$x$ is $N$-good but not $N$-super-good}]\le e^ {-cN^ d}.
\end{equation}
 By Lemma \ref{lem:contgoodbox}, there exists $N\ge 1$ large enough such that
\[\P[\text{$x$ is $N$-super-good}]\ge 1-\delta.\]
Let $\mathcal{C}_0(p)$ denote the cluster of $0$ in $\pazocal{V}(p)$.
Assume that $\mathcal{C}_0(p)$ is finite and $\mathrm{diam} \mathcal{C}_0(p)\ge N$. Let $C$ be the set of sites such that the corresponding boxes intersect $\mathcal{C}_0(p)$, that is 
\[C\coloneqq\{x\in\Z^ d: \mathcal{C}_0(p)\cap \Lambda_N(2Nx)\ne \emptyset\}.\]
Denote $\partial^{int} C$ be the set of sites in $x\in C$ with a neighbour which is connected to infinity in $\Z^d \setminus C$.
We claim that there exist $N\geq1$ and $c>0$ such that for every $n\geq2$
\begin{equation}\label{eq:claimintbound}
    \P[|\partial^{int} C|\ge n]\le e^ {-cn}.
\end{equation}
To conclude the proof from the previous inequality, it is sufficient to note that there exists $c_0>0$ depending on $N$ such that if $\mathrm{diam} (\mathcal{C}_0(p))\ge R$ then one has $|\partial^{int} C|\ge c_0 R$, and if $\Vol(\mathcal{C}_0(p))\ge R$ then, using a standard isoperimetric inequality on $\Z^d $, one has $|\partial^{int} C|\ge c_0 R^ {\frac{d-1}d}$.

We now turn to the proof of \eqref{eq:claimintbound}.
It is easy to check that if $ |\partial^{int} C|\ge 2$, then $\mathrm{diam} \mathcal{C}_0(p)\ge N$ and every $x\in\partial^{int} C$ is $N$-bad as there is at least one sub-box of $\Lambda_{3N}(2Nx)$ not intersecting $\mathcal{C}_0(p)$. In particular, every $x\in\partial^{int} C$ is $N$-super-bad. Moreover, by for instance \cite[Theorem 5.1]{timar2007cutsets}, the set $\partial^{int} C$ is $*$-connected. 
Let $\Gamma$ be a $*$-connected set. There exists $\Gamma'$ such that $|\Gamma'|\ge c_d|\Gamma|$ and for all $x\ne y \in \Gamma '$, we have $\Lambda_{4(1+2\sqrt d)N}(2Nx)\cap \Lambda_{4(1+2\sqrt d)N}(2Ny)=\emptyset $.
It follows that 
\begin{equation}
    \P[\forall x\in\Gamma \quad \text{$x$ is $N$-super-bad}]\le \P[\forall x\in\Gamma' \quad \text{$x$ is $N$-super-bad}]\le \delta ^ {c_d|\Gamma|}.
\end{equation}
 Inequality \eqref{eq:claimintbound} follows easily by a union bound over all possible sets $\partial^{int} C$. Note that since $C$ contains $0$, if $\partial ^{int}C$ is of size $k$ then it has to intersect the box $\Lambda_k$. By union bound, we get
 \begin{equation*}
       \P[|\partial^{int} C|\ge n]\le \sum_{k\ge n}\sum_{x\in \Lambda_k}\sum_{\Gamma \in {\rm Animals}_x^ k}\P[ \text{$x$ is $N$-bad} \,\, \forall x\in \Gamma ]\le \sum_{k\ge n}(2k)^ d7^ {dk}\delta ^ {c_dk}\le e^ {-cn}
 \end{equation*}
 where ${\rm Animals}_x^ k$ is the set of $*$-connected component containing $x$ of size $k$. We used in the second inequality that $|{\rm Animals}_x^ k|\le 7^ {dk}$ (see for instance \cite[(4.24) p81]{Gri99a}) and we choose $\delta$ small enough depending on $d$ such that the last inequality holds.
\end{proof}

\begin{remark}\label{rk:section 4}
Let us briefly explain how one could prove that there exists $c>0$ such that for all $L\ge 1$
    \begin{equation}
        \P[\pazocal{U}(L,p)^{\mathrm{c}}]\le e ^ {-cL}.
    \end{equation}
On the event $\pazocal{U}(L,p)^{\mathrm{c}}$, there are at least two disjoint clusters $C_1$ and $C_2$ crossing the annulus $\Lambda_L\setminus \Lambda_{L/2}$. Fix $N\ge1$ be large enough and consider the macroscopic lattice where a site $x\in\Z^d$ corresponds to the box $\Lambda_N(2Nx)$. Define $U$ to be the set of macrosocopic sites such that their corresponding box intersects $C_1\cap \Lambda_L$. Let $V$ be the set of macrosocopic sites such that the corresponding box intersects $\Lambda_L\cap C_2$ or intersects $\Lambda_L$ but not $C_1$. Both $U$ and $V$ contains a crossing of $\Lambda_{L/2N}\setminus \Lambda_{L/4N}$.  From Lemma \ref{lem:crossing_boundary}, we can find a macroscopic $*$-path of length at least $L/4N$ such that any element of the path is at distance at most $1$ from $U$ and $V$. It is easy to check that all the sites in this path are $N$-bad. The remaining of the proof is then very similar to the proof of Theorem \ref{thm:sharpness}.
\end{remark}

\section{Grimmett--Marstrand for $\pazocal{V}_N$}\label{sec:Grim_Mars}

In this section we prove Theorem~\ref{thm:disc_decay_truncated}. The proof is an adaptation of the argument of Grimmett--Marstrand \cite{GriMar90}. The technicalities of the proofs are specific to Voronoi percolation, but the adaptation of \cite{GriMar90} should be possible for any reasonable percolation model with finite range dependence and satisfying a sprinkled finite energy property (see Proposition~\ref{prop:domsto} below).

Denote $ \mathrm{Ber}(\delta)$ the distribution of a Bernoulli random variable of parameter $\delta$.
Let $\omega_\delta \sim_d \mathrm{Ber}(\delta)^ {\otimes \Z^ d}$ independent of $\eta$ and define
\[\omega_\delta ^N\coloneqq \bigcup_{x\in\Z^ d: (\omega_\delta)_x=1}\overline{\Lu_N(Nx)}\]
where $\overline A$ denotes the closure of the set $A$.
In this section, we will prove the two following propositions from which Theorem~\ref{thm:disc_decay_truncated} follows readily.
\begin{proposition}\label{prop:domsto}
    For every $N\ge1$ and $\ep>0$, there exists $\delta=\delta(N,\ep)>0$ such that for every $p\in(\ep,1-\ep)$ the following stochastic dominations hold
    \begin{align}
    \label{eq:domsto_1}
    \pazocal{V}_N(p-\varepsilon) &\preceq \pazocal{V}_N(p)\setminus \omega^N_\delta,\\
    \label{eq:domsto_2}
     \pazocal V_N(p+\ep) &\succeq \pazocal V_N(p)\cup \omega_\delta ^N.
    \end{align}
\end{proposition}

\begin{proposition}\label{prop:interGM}
  Fix $N\geq1$ and denote by $p_c(N)$ the percolation critical point of $\pazocal{V}_N$ defined as in \eqref{eq:def_pc}. Then for every $p>p_c(N)$ and $\delta>0$, there exists $c=c(p,N,\delta)>0$ such that for $R$ large enough
\begin{equation}\label{eq:disc_decay_truncated2} 
    \P[\lr{}{\pazocal{V}_N(p)\cup \omega_\delta ^N }{B_R}{\infty}]\ge 1-e^ {-cR^ {d-1}}.
\end{equation}  
\end{proposition}

The proof of Proposition~\ref{prop:domsto} is presented in Section~\ref{subsec:domsto}. We now explain the main ingredients in the proof of Proposition~\ref{prop:interGM}. Let us introduce some useful notations.
Define for $n,m\ge 1$, the ``corners''
\[T_n\coloneqq\{x\in\partial \Lambda_n: x_1=n,\,  x_i\ge 0 \,\, \forall i\in\{2,\dots,d\}\},\]
\[T_{m,n}\coloneqq\{x+t\mathbf e_1: x\in T_n,\, t\in[0,2m]\}.\]
Consider the random set of points in $T_n$ included in a ``seed'' in $T_{m,n}$, i.e.~a translate of $\Lambda_m$ in $T_{m,n}$  which is fully open in $ \omega_\delta ^N$, namely
\[K_{m,n}( \omega_\delta ^N)\coloneqq\{x\in T_n: \exists z\in \R^ d\quad \Lambda_m(z)\subset   \omega_\delta ^N\cap T_{m,n},\, x\in \Lambda_m(z)\}.\]

The proof strategy of Proposition~\ref{prop:interGM} follows closely the original one of Grimmett \& Marstrand \cite{GriMar90}. Though due to the short range dependence and the continuous setting of the model, the proofs require nontrivial technical adaptations.
The proof is done in three main steps. The first step consists in proving that we have a good probability to connect to a seed.

\begin{lemma}\label{lem:seed} 
For every $N\ge 1$, $p>p_c(N)$ and $\varepsilon,\delta >0$, there exist $n,m\in N\N$, $n\ge m \ge 2N$ such that
    \[\P\left[\lr{\Lambda_n}{\pazocal{V}_N(p)}{\Lambda_m}{ K_{m,n}(\omega_\delta ^N)}\right]\ge 1-\varepsilon.\]
\end{lemma}

The proof of Proposition~\ref{prop:interGM} uses a dynamic exploration process starting from the ball $B_R$ in such a way that if the exploration process never ends then $B_R$ is connected to infinity. 
In order to deal with the negative information that may arise from this exploration process (coming from the closed regions on the boundary of the exploration), we will need the following lemma to ensure that up to a sprinkling of the parameter the explored region is connected to a seed.
\begin{lemma}\label{lem:sprinklingseed}
    For every $N\ge 1$, $p>p_c(N)$ and $\varepsilon,\delta >0$, there exist $n,m\in N\N$, $n\ge m \ge 2N$ such that the following holds. 
    For $\Lambda_m\subset \Re\subset \R^ d$ and $x\in N\Z^d$,
    denote 
    $$G_\Re(x)\coloneqq\{\Lc_{4N}(x)\subset \omega_\delta ^ N \}\cap \big( \{\lr{\Lambda_{n}\setminus (\Re+ \Lc_{2N})}{\pazocal{V}_N(p)}{\Lc_{4N}(x)}{ K_{m,n}(\omega_\delta ^ N)}\}\cup \{\Lc_{4N}(x)\cap K_{m,n}(\omega_\delta ^ N)\ne \emptyset \} \big).$$
    Then, for every $\Lambda_m\subset \Re\subset \R^ d$, we have
    \begin{equation*}
        \P\Big[ \bigcup_{x\in N\Z^d\cap (\Re+\Lc_N)}\, G_\Re(x) \Big]\ge 1-\ep.
    \end{equation*}
\end{lemma}

\begin{remark}
    In Lemmas~\ref{lem:seed} and \ref{lem:sprinklingseed}, we require $m$ and $n$ to be multiples of $N$ in order to use the invariance of $\pazocal{V}_N(p)$ with respect to the symmetries of $N\Z^d$. In particular, this implies that the statements still hold if the corner $T_n$ (and $T_{m,n}$) is replaced by any of the $d2^d$ possible corners of $\Lc_m$.
\end{remark}

\begin{figure}[!h]
    \centering
    \includegraphics[width=0.5\textwidth]{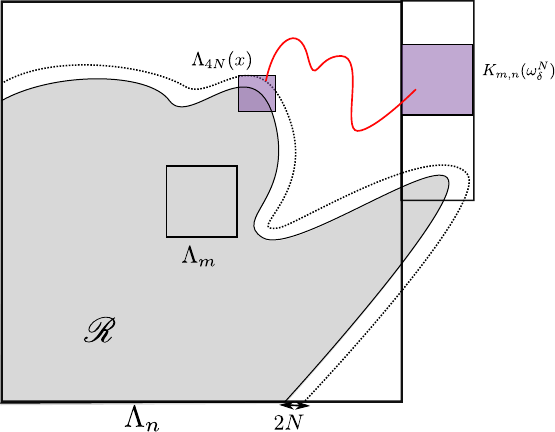}
    \caption{\label{sprinklingseed} The event $G_\Re(x)$. They grey region represents $\Re$, while the dotted region represents the boundary of $\Re+\Lc_{2N}$. The connection in red is in $\mathcal{V}_N(p)\cap (\Lc_n\setminus (\Re+\Lc_{2N}))$. The purple regions are open in $\omega_\delta^ N$.}
\end{figure}

The last step is to prove that this exploration process stochastically dominates the exploration process of supercritical Bernoulli site percolation. Here, the sites correspond to boxes of side-length $2(n+m)$. We will say that a neighbouring site of the explored region is occupied if after sprinkling the explored region is connected to a seed inside this box. In this step, we iterate the construction of connecting to a seed. Thanks to Lemma~\ref{lem:sprinklingseed}, each step has good probability of occurring. We conclude that it stochastically dominates the exploration process of supercritical Bernoulli site percolation.

We postpone the proof of Propositions~\ref{prop:domsto} and Lemmas~\ref{lem:seed} and \ref{lem:sprinklingseed} to the next subsections.

\begin{proof}[Proof of Proposition~\ref{prop:interGM}]
Fix $p>p_c(N)$, $\delta>0$ and $R$ large enough. 
Take $\varepsilon>0$ such that
$1-2\varepsilon> p_c^{site}(\Z^d)$ and let $n\geq m\geq 1$ be as in Lemma~\ref{lem:sprinklingseed}. 
Let $\omega_t^1, \omega_t^2$, for integers $t\ge1$, be an i.i.d.~family of random variables distributed as $\omega_\delta ^N$ and independent of $\eta$. 
We will run an exploration process where we progressively reveal $\overline \eta ^ N(p)$ and $\omega^i_t$ on boxes of the form $\Lu_N(x)$ for $x\in N\Z^ d$.
Set $n'\coloneqq n+m$. For each site $x\in\Z^d$, we identify it to the box $\Lambda_{n'}(2n'x)$. 
Set 
\[A_0\coloneqq \{x\in \Z^d : \Lambda_{n'}(2n'x)\subset B_R\},\quad E_0\coloneqq\emptyset,\quad \pazocal V_0\coloneqq \pazocal V_N(p).\]
We will inductively construct a process $(A_t,E_t,\pazocal V_t)_{t\geq0}$. For every $t\ge 0$, we denote by $\pazocal C_t$ the cluster of $B_R$ in $\pazocal V_t$ (more precisely, this means the union of the clusters in $\pazocal V_t$ intersecting $B_R$). The construction will be done in such a way that for every $x\in A_t$, there exists a point $w$ such that $\Lambda_m(w)\subset \Lambda_{n'}(2n'x)\cap \pazocal C_t$. In other words, in each box corresponding to an active site the cluster is connected to a seed. Moreover, the construction will ensure that $\pazocal V_t\subset \pazocal V_{t+1}$ and so $\pazocal C_t\subset \pazocal C_{t+1}$.

Let $t\ge 0$. We build $A_{t+1},E_{t+1},\pazocal V_{t+1}$  from $A_t,E_t,\pazocal V_t$ as follows.
If there is a vertex in $(A_t\cup E_t)^{\mathrm{c}}$ neighboring $A_t$, then we stop the exploration and set $(A_{t+1},E_{t+1},\pazocal V_{t+1})=(A_t,E_t,\pazocal V_t)$. 
Otherwise, let us denote such a vertex in $(A_t\cup E_t)^{\mathrm{c}}$ by $x_{t}$ (if there are several such vertices, we choose the earliest one according to a deterministic ordering).
We set 
\[\pazocal V_{t+1}=\pazocal V_t\cup  \big((\omega_t^1\cup\omega_t^2)\cap \Lambda_{3n'}(2n'x_t)\big).\] 
We say that the step is successful if there exists a point $v$ such that $\Lambda_m(v)\subset \Lambda_{n'}(2n'x_t)\cap \pazocal C_{t+1}$.
If this step is successful, then we set 
\[A_{t+1}=A_t\cup\{x_t\},\quad E_{t+1}=E_t.\]
Otherwise, we set 
\[A_{t+1}=A_t,\quad E_{t+1}=E_t\cup \{x_t\}.\]

\begin{figure}[!h]
    \centering
    \includegraphics[width=0.65\textwidth]{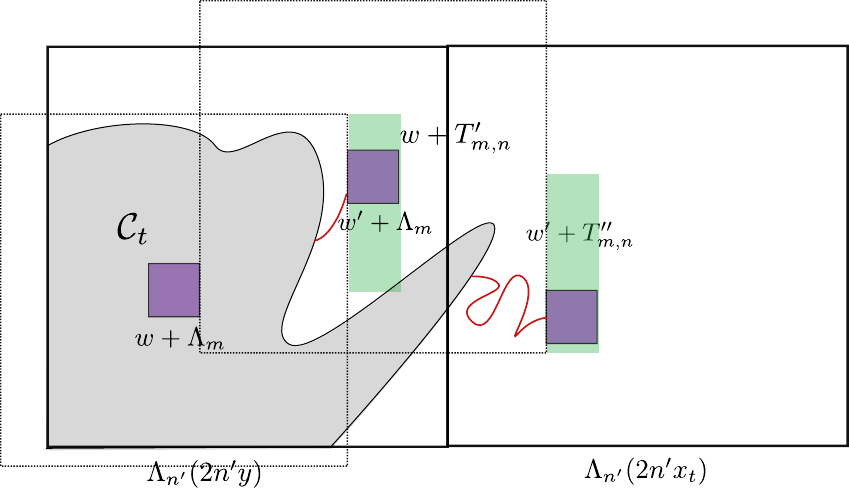}
    \caption{\label{steering} The steering procedure. Connections in $\pazocal{V}_N(p)$ are represented in red. Purple boxes are seeds, i.e.~open boxes in some $\omega_t^i$. The green regions are the two target corners.}
\end{figure}

We claim that, conditionally on the whole history of the process, each step is successful with probability at least $1-2\ep $, namely: 
\[\P[A_{t+1}=A_t\cup\{x_t\} \,|\, (A_s,E_s,\pazocal{C}_s)_{0\leq s\leq t}]\ge 1-2\varepsilon \quad \text{a.s..}\]
Let $y\in A_t$ be a neighbor of $x_t$.
By construction of the exploration, there exists a point $w$ such that $\Lambda_m(w)\subset \Lambda_{n'}(2n'y)\cap \pazocal C_t$. 
Let us now explain how by a steering procedure, we can connect $\pazocal{C}_t$ to a seed in $\Lambda_{n'}(2n'x_t)$ with high probability after a two-steps sprinkling -- see Figure~\ref{steering}. It is easy to check that on the face corresponding to the direction $x_t-y$, there exists a corner $T'_{m,n}$ (that is one of the $2^{d-1}$ regions symmetric to $T_{m,n}$ on that face) such that $w+T'_{m,n}\subset \Lambda_{n'}(2n'x_t)\cup \Lambda_{n'}(2n'y)$. 
We will now use Lemma~\ref{lem:sprinklingseed} to connect $w$ to a seed in $(w+T'_{m,n})\cap \omega_t^1$. We define $\Re$ to be the union of the set of $N$-boxes at distance less than $N$ from $\pazocal C_t$. In particular, in the exploration of $\pazocal{C}_t$, we have not revealed anything outside $\Re$. Note that the event $G_{\Re}(x)$ defined in Lemma~\ref{lem:sprinklingseed}, with $\omega^N_\delta$ replaced by $\omega^1_t$, is independent of the configuration inside $\Re$. Moreover, we have for any $x\in \Re $, $\Lambda_{4N}(x) \cap \pazocal C_t\ne \emptyset$. Thanks to Lemma~\ref{lem:sprinklingseed}, the probability  that $\pazocal{C}_t$ is connected to a seed $\Lc_m(w')\subset (w+T'_{m,n})\cap \omega_t^1$ in $\pazocal{V}_t\cup (\omega_t^ 1\cap \Lambda_{3n'}(2n'x_t))$ is at least $1-\ep$. Conditionally on the existence of such a $w'$, we can now repeat the same procedure centered at $w'$ by choosing an appropriate corner of $T''_{m,n}$ such that $w'+T''_{m,n}\subset \Lc_{n'}(2n'x_t)$. By another application of Lemma~\ref{lem:sprinklingseed} as before, with $\pazocal{C}_t$ replaced by the cluster of $B_R$ in $\pazocal V_t \cup (\omega_t^ 1\cap \Lambda_{3n'}(2n'x_t))$, we can ensure that the seed centered at $w'$ is connected in $\pazocal{V}_t\cup ((\omega_t^ 1\cup \omega_t^2)\cap \Lambda_{3n'}(2n'x_t))$ to some seed $\Lambda_m(w'') \subset \Lambda_{n'}(2n'x_t)\cap \omega^2_t$.
We conclude that the step is successful with probability at least $(1-\ep)^ 2\ge 1-2\ep>p_c^ {site}$.

Finally, it is clear that each $n'$-box is sprinkled at most $2.3^d$ times. Hence we have the stochastic domination 
\[\pazocal V_\infty\preceq \pazocal V_N(p)\cup \omega_{\delta'}^N,\]
where $\pazocal{V}_\infty\coloneqq \cup_{t\ge 1}\pazocal V_t$ and $\omega_{\delta'}^N$ is independent of $\pazocal{V}_N(p)$, with $\delta'=2.3^d\delta$. 
The conclusion of the theorem follows from the lemma below, which is an easy adaptation of \cite[Lemma 1]{GriMar90}.
\begin{lemma}\label{lem:explor}Let $X_t = (A_t, E_t)$ be a random exploration sequence of $\Z^d$. Assume that there exists $q\in[0,1]$ such that for every $t\ge 0$
    \[\P[A_{t+1}=A_t\cup\{x_t\}|X_0,\dots,X_t]\ge q\quad \text{a.s.,}\]
where $x_t$ is the earliest vertex in $(A_t \cup E_t)^{\mathrm{c}}$ neighboring $A_t$. Then $A_\infty\coloneqq \cup_{t\ge 0}A_t$ stochastically dominates the cluster of $A_0$ in Bernoulli site percolation of parameter $q$ on $\Z^d$. In particular, for $q>p_c^{site}(\Z^d)$ there exists $c = c(q) >0$ such that if $\Lc_k\subset A_0$ then
\[\P(\lr{}{A_\infty}{\Lc_k}{\infty})\ge 1-e^ {-ck^{d-1}}.\]
\end{lemma}
\end{proof}

\subsection{Proof of Lemma~\ref{lem:seed}}

We split the proof of Lemma~\ref{lem:seed} into two steps. In the first step, we prove that there exists $n\ge m$ such that $\Lambda_m$ is connected to a lot of different points close to the boundary of $\Lambda_n$. Second, we prove that by a sprinkling we can open a seed in $T_{m,n}$ that is connected to $\Lambda_m$. The following lemma corresponds to the first step.
\begin{lemma}\label{lem:seed1} For every $N\ge 1$, $p>p_c(N)$ and $\varepsilon>0$, there exists $m\in N\N$ such that the following holds. For any $\kappa\ge 1$, there exists $n\in N\N$, $n\ge m$, such that
    \[\P\left[ |\{x\in N\Z^d:\, \Lu_N(x)\cap T_n\ne \emptyset,\, \lr{\Lambda_n}{\pazocal{V}_N(p)}{\Lambda_m}{\Lu_N(x)}\}| \ge \kappa\right]\ge 1-\varepsilon.\]
\end{lemma}
\begin{proof}[Proof of Lemma \ref{lem:seed1}] Fix $N\ge 1$, $p>p'>p_c(N)$ and let $\delta$ be as in Proposition \ref{prop:domsto} such that
\begin{equation}\label{eq:domsto_lem:seed1}
    \pazocal{V}_N(p') \preceq \pazocal{V}_N(p)\setminus \omega^N_\delta.
\end{equation}
Given a $\ep_0>0$ to be chosen later, let $m\in N\N$ be large enough such that 
 \[\P[\lr{}{\pazocal{V}_N(p')}{\Lambda_m}{\infty}]\ge 1-\ep_0.\]
 Let $\kappa \ge 1$.
Denote by 
\[\mathcal A_k\coloneqq \{x\in \Z^ d: \mathring \Lu_N(Nx)\subset  \Lambda_{kN}\setminus\Lambda_{(k-1)N} \}\]
where $\mathring A $ denotes the interior of the set $A$.
We start by proving that for $K$ large enough depending on $\delta,m,\varepsilon$, there are many vertices in $\bigcup_{m\le k\le K}\cA_k$ such that the corresponding box is connected to $\Lambda_m$ in $\pazocal V_N(p)$.
Note that for any $K\ge 2m$
\begin{equation}\label{eq:sum telescopic}
    1\ge  1- \P[\lr{\Lambda_{KN}}{\pazocal{V}_N(p)\setminus \omega^N_\delta} {\Lambda_m}{\partial \Lambda_{KN}}]\ge \sum_{k=m}^{K-1}\P[\lr{\Lambda_{kN}}{\pazocal{V}_N(p)\setminus \omega^N_\delta} {\Lambda_m}{\partial \Lambda_{kN}}]-\P[{\lr{\Lambda_{(k+1)N}}{\pazocal{V}_N(p)\setminus \omega^N_\delta} {\Lambda_m}{\partial \Lambda_{(k+1)N}}}].
\end{equation}
Denote by $V_k$ the set of vertices in the annulus $ \cA_{k+1}$ such that the corresponding box is connected to $\Lambda_m$ in $(\pazocal{V}_N(p)\setminus \omega^N_\delta)\cap \Lambda_{kN} $, that is
\[V_k\coloneqq \left\{x\in\cA_{k+1}:\lr{\Lambda_{kN}}{\pazocal{V}_N(p)\setminus \omega^N_\delta}{\Lambda_m}{\Lu_N(Nx)}\right\}.\]
Note that the set $V_k$ does not depend on the values of $\omega_\delta$ on $\cA_{k+1}$.
It is easy to check that
\[\{V_k\neq \emptyset\} \cap \{(\omega_\delta)_x=1 ~~ \forall x\in V_k\}\subset \{\lr{\Lambda_{(k+1)N}}{\pazocal V_N(p)\setminus\omega^N_\delta} {\Lambda_m}{\partial \Lambda_{kN}}\} \setminus \{\lr{\Lambda_{(k+1)N}}{\pazocal V_N(p)\setminus\omega^N_\delta} {\Lambda_m}{\partial \Lambda_{(k+1)N}}\}.\] 
Set $N_k\coloneqq|V_k|$.
It follows that
\[\P[\lr{\Lambda_{kN}}{\pazocal{V}_N(p)\setminus \omega^N_\delta} {\Lambda_m}{\partial \Lambda_{kN}}]-\P[{\lr{\Lambda_{(k+1)N}}{\pazocal{V}_N(p)\setminus \omega^N_\delta} {\Lambda_m}{\partial \Lambda_{(k+1)N}}}]\ge \P[ V_{k}\ne \emptyset, \forall x\in V_{k}\quad (\omega_\delta)_x=1]\ge  \E[ \delta ^{N_k}\mathbf{1}_{N_k>0}].\]
Combining the previous inequality together with inequality \eqref{eq:sum telescopic} yields
\[\sum_{k=m}^{K-1}\E\left [\delta ^{N_k}\mathbf{1}_{N_k>0}\right]\le 1.\]
In particular, there exists $L\in\{m,\dots,K-1\}$ such that
\[\E\left [\delta ^ {N_L}\mathbf{1}_{N_L>0}\right]\le \frac 2 K.\]
Set $c_0\coloneqq\frac 1{ 2|\log \delta |}$.
It yields
\begin{equation*}
   \P[0<N_L\le c_0\log K] \delta^ {-c_0\log K}\le \E\left [\delta ^ {N_L}\mathbf{1}_{N_L>0}\right]\le \frac 2 K
\end{equation*}
and
\[\P[0<N_L\le c_0\log K]\le \frac 2 {\sqrt K}.\]
Moreover, by stochastic domination \eqref{eq:domsto_lem:seed1} we have
\[ \P[N_L=0]=1-\P\left[\lr{\Lambda_{LN}}{\pazocal{V}_N(p)\setminus \omega^N_\delta}{\Lambda_m}{\partial \Lambda_{LN}}\right]\le 1- \P[\lr{}{\pazocal{V}_N(p')}{\Lambda_m}{\infty}]\le \ep_0 .\]
Finally, we get for $K$ large enough
\begin{equation*}
     \P[N_L> c_0\log K]\ge 1-  \P[0<N_L\le c_0 \log K]- \P[N_L=0]\ge 1- \frac 2{\sqrt K}-\ep_0\ge 1-2\ep_0.
\end{equation*}
Set $n\coloneqq N(L+1)$ and
\[\cE\coloneqq \left\{|\{x\in N\Z^d:\, \Lu_N(x)\cap T_n\ne \emptyset,\, \lr{\Lambda_n}{\pazocal{V}_N(p)}{\Lambda_m}{\Lu_N(x)}\}| \ge \kappa\right\}.\]
Using the square root trick applied to the $d2^d$ symmetric copies of $T_n$, we obtain 
\[ \P[N_L> d2^d \kappa]\le 1-\P[\cE^{\mathrm{c}}]^ {d2^ d}.\]
By choosing $K$ large enough such that $c_0\log K>d2^d \kappa$, it yields
\[\P[\cE^{\mathrm{c}}]\le (2\ep_0)^ {\frac 1 {d2^ d}} .\]
The result follows by choosing $\ep_0=\ep^{d2^d}/2$.
\end{proof}

\begin{proof}[Proof of Lemma~\ref{lem:seed}]  Let $m\ge 2N$. Let $\kappa\ge 1$ to be chosen later in the proof. Let $n\ge m$ be as in the statement of Lemma \ref{lem:seed1}. We have
  \begin{equation}\label{eq: cont ce}
   \P\left[ |V| \ge \kappa\right]\ge 1-\varepsilon,
 \end{equation}
where 
\[V\coloneqq \left\{x\in N\Z^d:\, \Lu_N(x)\cap T_n\ne \emptyset,\, \lr{\Lambda_n}{\pazocal V_N(p)}{\Lambda_m}{\Lu_N(x)}\right\}.\]
Let us now prove that up to a sprinkling, we can connect to a point in $K_{m,n}(\omega_\delta ^N)$. To do so, we need to fully open a seed in $T_{m,n}$ close to a box $\Lu_N(x)$ for $x\in V$.
By pigeon-hole principle, there exists $V'\subset V$ such that 
\[\forall x,y\in V' \quad \|x-y\|_\infty \ge 6m\quad \text{and}\quad |V'|\ge c_d \frac{|V|}{m^ d},\]
 where $c_d$ is a constant depending only on $d$. We choose this set according to some deterministic rule.
Thanks to this choice, we have
for $x,y\in V'$ that 
$\Lambda_{2m}(x)\cap \Lambda_{2m}(y)=\emptyset$ and for every $w\in N\Z^ d$ such that $\Lu_N(w)\cap\Lambda_{m}(x)\ne \emptyset $ then  $\Lu_N(w)\cap\Lambda_{m}(y)= \emptyset $.
We define the events 
\[\cE(x)\coloneqq \{ \Lambda_{m}(x)\subset \omega_\delta ^N\}.\]
Since the vertices in $V'$ are sufficiently spaced, we have that conditionally on $\overline \eta^N(p)$, the events $(\cE(x))_{x\in V'}$ are independent.
It is easy to check that 
\[\{V\ne\emptyset\}\cap \bigcup_{x\in V'}\cE(x) \subset \{\lr{\Lambda_n}{\pazocal{V}_N(p)}{\Lambda_m}{ K_{m,n}(\omega_\delta ^N)}\}.\]
For short, write $\cE\coloneqq \{|V|\ge \kappa\}$.
Hence, it yields
\begin{equation}\label{eq:seed2}
    \P[\lr{\Lambda_n}{\pazocal{V}_N(p)}{\Lambda_m}{ K_{m,n}(\omega_\delta ^N)}]
\ge 1-\P[\cE^{\mathrm{c}}] - \E\left[\mathbf{1}_\cE \prod_{x\in V'}\P[\cE(x)^{\mathrm{c}}|\overline \eta ^N (p)]\right].
\end{equation}
Note that $\cE$ is measurable with respect to $\overline \eta ^N (p)$.
Since $\P[\cE(x)]\geq \delta ^ {Cm^d}$, we have
\begin{equation*}
    \E\left[\mathbf{1}_\cE \prod_{x\in V'}\P[\cE(x)^{\mathrm{c}}|\overline \eta ^N (p)]\right]
\le \E[\mathbf{1}_\cE (1-\delta ^ {Cm^ d} )^{|V'|}]\le (1- \delta ^ {Cm^ d})^{ c_dm^{-d}\kappa}\le \ep
\end{equation*}
where $\kappa$ is chosen large enough depending on $\ep,\delta,N$ and $m$.
The result follows by combining the previous inequality together with inequalities \eqref{eq: cont ce} and \eqref{eq:seed2}.
\end{proof}

\subsection{Proof of Lemma~\ref{lem:sprinklingseed}}
Similarly to Lemma~\ref{lem:seed}, the proof of Lemma~\ref{lem:sprinklingseed} 
consists of two steps. First, we prove that there are a lot of boxes intersecting $\Re$ that are either very close to a seed or such that the box is connected to a seed outside $\Re$. Second, we argue that with high probability there is at least one of these boxes whose neighborhood is open in the $\delta$-sprinkling.

\begin{proof}[Proof of Lemma~\ref{lem:sprinklingseed}]
Fix $p>p'>p_c(N)$ and let $\delta_0$ be as in Proposition \ref{prop:domsto} such that
\begin{equation}\label{eq:domstolem5.6}
    \pazocal{V}_N(p') \preceq \pazocal{V}_N(p)\setminus \omega^0,
\end{equation}
where $\omega^0\sim_d\omega^N_{\delta_0}$ is independent of $\pazocal V_N(p)$. Given $\delta>0$, let $\delta'>0$ be such that $\delta'(2-\delta')=\delta$ and consider $\omega^1,\omega^2 \sim_d\omega^N_{\delta'}$ such that $\pazocal{V}_N(p)$, $\omega^0$, $\omega^1$ and $\omega^2$ are all independent. Note that by construction $\omega:=\omega^1\cup\omega^2 \sim_d \omega^N_\delta$ is independent of $\pazocal{V}_N(p)$.
Let $\ep_0>0$ to be chosen later depending on $\delta $, $\ep$ and $\delta_0$. Let $m\le n$ be as in Lemma~\ref{lem:seed} such that 
  \[\P\left[\lr{\Lambda_n}{\pazocal{V}_N(p')}{\Lambda_m}{ K_{m,n}(\omega^1)}\right]\ge 1-\ep_0.\]
Fix any $\Lambda_m\subset \Re\subset \R^ d$ and set
  \[V\coloneqq \left\{x\in N\Z^d\cap (\Re+\Lc_N):\, \lr{\Lambda_{n}\setminus (\Re+ \Lambda_{2N})}{\pazocal{V}_N(p)}{\Lambda_{4N}(x)}{ K_{m,n}(\omega^1)} \text{ or }\Lambda_{4N}(x)\cap K_{m,n}(\omega^1)\ne \emptyset \right\}.\]
Note that the set $V$ does not depend on $\omega_{\delta_0}$. Let $\widetilde V$ be the boxes in a neighborhood of $V$, that is
\[\widetilde V\coloneqq \{w\in N\Z^d :\exists x\in V\quad\Lu_N(w)\subset \Lambda_{4N}(x)\}.\]
Note that $|\tilde{V}|\leq 8^d|V|$. Set $c\coloneqq \frac 1 {2\cdot 8^ d|\log \delta_0|}$.
Let us first prove that
\begin{equation}\label{eq:step1toprove}
    \P\left[|V|\ge c|\log \varepsilon_0|\right]\ge 1-\sqrt{\varepsilon_0}.
\end{equation}
We claim that
\begin{equation}\label{eq:claiminclusion}
    \big\{ \bigcup_{w\in \tilde{V}} \overline{\Lu_N(w)} \subset \omega^0 \big\} \subset  \{\nlr{\Lambda_n}{\pazocal{V}_N(p)\setminus \omega^0}{\Lambda_m}{ K_{m,n}(\omega^1)}\}.
\end{equation}
Assume that $\lr{\Lambda_n}{\pazocal{V}_N(p)\setminus \omega^0}{\Lambda_m}{ K_{m,n}(\omega^1)}$. We will prove that there exist $w\in \tilde{V}$ such that $\overline{\Lu_N(w)} \not\subset \omega^0$, thus yielding \eqref{eq:claiminclusion}.
Let $x_1,\dots,x_{\ell}\in\overline\eta_o^N(p)$ be the sequence of open points corresponding to the successive open cells crossed by a path in $\pazocal{V}_N(p)\setminus \omega^0$ from $\Lambda_m$ to $ K_{m,n}(\omega^1)$ in $\Lambda_n$. 
We consider two cases. First assume that $x_\ell \in \Re+\Lambda_{2N}$. Then, there exists $x\in N\Z^ d\cap (\Re+\Lc_N)$ such that $x_\ell \in x+\Lambda_{3N}$. Since $(x_\ell +B_N)\cap  K_{m,n}(\omega_\delta ^N)\ne \emptyset$, it follows that $\Lambda_{4N}(x)\cap K_{m,n}(\omega_\delta ^ N)\ne \emptyset$ and $x\in V$. Considering $w\in N\Z^ d$ such that $x_\ell \in\Lu_N(w)$, it follows that $w\in\widetilde V$ and $\overline{\Lu_N(w)} \not\subset \omega^0$. 
Let us now assume  that $x_\ell \notin \Re+\Lambda_{2N}$. Consider
 \[\ell_0\coloneqq\sup\{k\ge 1: x_k\in \Re+\Lambda_{2N}\}+1.\]
Hence, the path $x_{\ell_0},\dots, x_{\ell}$ remains inside $\Lambda_n\setminus (\Re+\Lambda_{2N})$. Furthermore, since $x_{\ell_0-1}\in \Re+\Lc_{2N}$ and $||x_{\ell_0}-x_{\ell_0-1}||\leq N$,  there exists $x\in N\Z^ d\cap (\Re+\Lc_N)$ such that $x_{\ell_0}\in \Lambda_{4N}(x)$. Hence we get $\lr{\Lambda_{n}\setminus (\Re+ \Lambda_{2N})}{\pazocal{V}_N(p)\setminus \omega^0}{\Lambda_{4N}(x)}{ K_{m,n}(\omega^1)}$. It follows that $x \in\widetilde V$ and $\overline{\Lu_N(x)} \not\subset \omega^0$. This finishes the proof of \eqref{eq:claiminclusion}.
As a consequence,
\begin{equation*}
    \begin{split}
       1-\ep_0\le  \P[\lr{\Lambda_n}{\pazocal{V}_N(p')}{\Lambda_m}{ K_{m,n}(\omega^1)}] &\le \P[\lr{\Lambda_n}{\pazocal{V}_N(p)\setminus\omega^0 }{\Lambda_m}{ K_{m,n}(\omega^1)}]\le 1-\P[\bigcup_{w\in \tilde{V}} \overline{\Lu_N(w)} \subset \omega^0]\\
        &\le 1-\E \Big[ \P\big[ \bigcup_{w\in \tilde{V}} \overline{\Lu_N(w)} \subset \omega^0 \,\big|\, \overline\eta ^N(p),\omega^1 \big]\mathbf{1}_{|V|\le c|\log \ep_0|} \Big]
        \\&\le 1-\delta_0^ {8^d c|\log \ep_0|}\P[|V|\le   c|\log \ep_0|]= 1- {\sqrt {\ep_0}} \P[|V|\le   c|\log \ep_0|].
    \end{split}
\end{equation*}
This yields inequality the desired inequality \eqref{eq:step1toprove}.

Let $V'\subset V$ such that $|V'|\ge \frac {|V|}{18^d}$ and for all $u,v\in V'$, $\|u-v\|_\infty \ge 9N$.
We define the event 
\[\cE(x)\coloneqq \{\Lambda_{4N}(x)\subset \omega^2\}.\]
By construction of $V'$ the events $\cE(x)$ are independent of each other and of $\pazocal V_N(p)$ and $\omega^1$. Moreover, we have
\[\P[\cE(x)|\pazocal V_N(p),\omega^1]\ge (\delta') ^ {8^ d}\]
and
\begin{equation*}
\begin{split}
    \P[\cap _{x\in V'}\cE(x)^{\mathrm{c}}, |V| \ge c|\log \varepsilon_0|]&=
    \E[\P[\cap _{x\in V'}\cE(x)^{\mathrm{c}}|\pazocal V_N(p),\omega^1]\mathbf{1}_{|V| \ge c|\log \varepsilon_0|}]\\
    &\le (1- (\delta')^{8^ d})^{c|\log \varepsilon_0|}.
    \end{split}
\end{equation*}
Recall that $\omega=\omega^1\cup\omega^2\sim_d \omega^N_\delta$ and set
\begin{equation*}
    G(x)\coloneqq\{\Lambda_{4N}(x)\subset \omega\}\cap (\{\lr{\Lambda_{n}\setminus (\Re+ \Lambda_{2N})}{\pazocal{V}_N(p)}{\Lambda_{4N}(x)}{ K_{m,n}(\omega)}\}\cup \{\Lambda_{4N}(x)\cap K_{m,n}(\omega)\ne \emptyset \}).
\end{equation*}
First note that 
\begin{equation*}
    \{\Lambda_{4N}(x)\subset \omega^2 \}\cap (\{\lr{\Lambda_{n}\setminus (\Re+ \Lambda_{2N})}{\pazocal{V}_N(p)}{\Lambda_{4N}(x)}{ K_{m,n}(\omega^1)}\}\cup \{\Lambda_{4N}(x)\cap K_{m,n}(\omega^1)\ne \emptyset \})\subset G(x).
\end{equation*}
Hence we get 
\begin{equation*}
\begin{split}
     \P\big[ \bigcup_{x\in N\Z^d\cap(\Re+\Lc_N)}G(x) \big]&\ge \P[|V| \geq c|\log \varepsilon_0|,\, \cup_{x\in V}\cE(x)]\\
     &\ge  1-\P[|V| \le c|\log \varepsilon_0| ] -\P[\cap _{x\in V'}\cE(x)^{\mathrm{c}}, |V| \ge c|\log \varepsilon_0|] \\
     &\ge 1-\sqrt {\ep_0} -\big( 1- (\delta')^{8^d} \big)^{c|\log \varepsilon_0|}.  
\end{split}
\end{equation*}
The result follows by choosing $\ep_0$ small enough depending on $\delta'$, $\ep$ and $\delta_0$ in such a way that
\[\sqrt {\ep_0} + \big( 1- (\delta')^{8^ d} \big)^{c|\log \varepsilon_0|}\le \ep.\] 
\end{proof}

\subsection{Proof of Proposition \ref{prop:domsto}}\label{subsec:domsto}

Let us start by introducing another coupling of $\pazocal V_N(p)$.
Let $\eta^o$ and $\eta^c$ be two independent Poisson point process of intensity $1$ on $\R^d\times [0,1]$. We denote the points in $\R^d\times[0,1]$ by $(x,t)$, with $x\in\R^d$ and $t\in [0,1]$. For every $p\in[0,1]$, define the set of $p$-open and $p$-closed points given by 
\begin{align}
\overline{\eta}_o(p)&\coloneqq\{x\in \R^d:\, \exists t\in[0,p] \text{ such that } (x,t)\in \eta^o\} ~~\text{ and}\\ 
\overline{\eta}_c(p)&\coloneqq\{x\in \R^d:\, \exists t\in(p,1] \text{ such that } (x,t)\in \eta^c\}.
\end{align}
We can then define $\overline{\eta}^N(p)=(\overline{\eta}^N_o(p),\overline{\eta}^N_c(p))$ and $\pazocal V_N(p)$ as in Section~\ref{sec:comparison} -- see \eqref{eq:def_eta_N} and \eqref{eq:def_truncated_Voronoi}.
The advantage of this coupling is that it is easier to open or close a region by sprinkling, as one can add new points as well as remove existing ones by changing $p$.

We start by proving the following two lemmas.

\begin{lemma}\label{lem:sprink} 
For every $N\ge 1$, $\ep>0 $ and $p\in(\ep,1-\ep)$, there exists $\delta>0$ such that for all $x\in N\Z^ d$, the following holds almost surely
    \begin{align}
        \label{eq:sprink_1}
        \P[\Lu_N(x)\subset \overline{\eta}^N_c(p-\ep)\,|\,\overline{\eta}^N(p)]&\ge \delta,\\
        \label{eq:sprink_2}
        \P[\Lu_N(x)\cap \overline{\eta}^N_c(p+\ep)=\emptyset\,|\,\overline{\eta}^N(p)]&\ge \delta.
    \end{align}
\end{lemma}

\begin{proof}
Fix $x\in N\Z^d$. We first prove \eqref{eq:sprink_1}. It follows easily from the following inequality
\begin{equation*}
\begin{split}
    \P&[\Lu_N(x)\subset \overline{\eta}^N_c(p-\ep)\,|\,\overline{\eta}^N(p)]\ge \P[ |\eta_c\cap (\Lu_N(x)\times [p-\ep,p))|\ge 2N^ d]=\P[\mathcal P(\ep N^ d)\ge 2N^d]
\end{split}
\end{equation*}
    where $\mathcal P(\lambda)$ is a Poisson random variable with parameter $\lambda$.

We now prove \eqref{eq:sprink_2}. We start by analysing the case in which $\Lu_N(x)$ contains too many closed points,
\begin{equation*}
    \begin{split}
        \P&[\Lu_N(x)\cap \overline{\eta}^N_c(p+\ep)=\emptyset\,|\, \overline{\eta}^N(p),\, \Lu_N(x)\subset\overline{\eta}^N_c(p)]\\
        &\ge  \P[\Lu_N(x)\cap \overline{\eta}^N_c(p+\ep)=\emptyset, |\Lu_N(x)\cap \overline{\eta}_c(p)|\le 3N^d \,\big|\, \overline{\eta}^N(p),\, \Lu_N(x)\subset\overline{\eta}^N_c(p)]\\
        &\ge \E\left [\left(\frac{\ep}{1-p}\right) ^{|\Lu_N(x)\cap \overline{\eta}_c(p)|}\mathbf{1}_{|\Lu_N(x)\cap \overline{\eta}_c(p)|\le 3N^d} \,\Big|\, \overline{\eta}^N(p),\, \Lu_N(x)\subset \overline{\eta}^N_c(p)\right ]\\
        &\ge \left(\frac{\ep}{1-p}\right) ^{3N^d}\P(\mathcal P((1-p)N^d)\le 3N^d| \mathcal P((1-p)N^d)\ge 2N^ d)\ge \frac 12 \left(\frac{\ep}{1-p}\right) ^{3N^d}.
    \end{split}
\end{equation*}
In the case where $\Lu_N(x)\cap \overline{\eta}^N_c(p)\ne \Lu_N(x)$, we have that $|\Lu_N(x)\cap \overline{\eta}_c(p)|\le 2N^d$. It yields that
\begin{equation*}
    \begin{split}
        \P&[\Lu_N(x)\cap \overline{\eta}^N_c(p+\ep)=\emptyset \,|\, \overline{\eta}^N(p),\,\Lu_N(x)\not\subset \overline{\eta}^N_c(p)]\ge \left(\frac{\ep}{1-p}\right) ^{2N^d}.
    \end{split}
\end{equation*}
The lemma follows from the two previous displayed inequalities.
\end{proof}

\begin{lemma}\label{lem:sprink_V_N}
For every $N\geq1$, $\ep>0$ and $p\in(\ep,1-\ep)$, there exists $\delta>0$ such that for all $x\in N\Z^ d$, the following holds almost surely 
\begin{align}
    \label{eq:sprink_V_N_1}\P[\Lu_N(x)\cap  \pazocal V_N(p-\ep) = \emptyset| \pazocal V_N(p)] &\ge \delta,\\
    \label{eq:sprink_V_N_2}\P[\Lu_N(x)\subset  \pazocal V_N(p+\ep)| \pazocal V_N(p)] &\ge \delta.
\end{align}
\end{lemma}

\begin{proof}
First notice that \eqref{eq:sprink_V_N_1} follows readily from \eqref{eq:sprink_1}. In order to prove \eqref{eq:sprink_V_N_2}, we will consider the event where the sprinkling enables to remove all the closed points in a neighborhood of the cube and add uniformly spread open points. When such an event occurs, it is easy to check that the cube is fully open.
Let us first prove that there exists $\delta>0$ depending on $N$, $p$ and $\ep$ such that for every $x\in N\Z^d$ and almost every $\overline \eta ^ N(p)\cap \Lu_N(x)$, we have
\begin{equation}\label{ineq:intermediate step}
    \P[ \cE(x), \Lu_N(x)\cap \overline \eta_c^N(p+\ep)= \emptyset \,|\, \overline \eta ^ N(p)]\ge\delta 
\end{equation}
where $\cE(x)$ is the event that there are sufficiently many uniformly spread points in $\overline{\eta}^N_o(p+\ep)\cap \Lu_N(x)$.
More formally, denote $\Lambda^{(1)}_N(x),\dots,\Lambda^{(d^d)}_N(x)$ the partition of $\Lu_N(x)$ into $d^d$ cubes of side-length $N/d$. Thanks to this choice of partition, the maximum distance between two points in a sub-cube is smaller than $N$.
We define
\[\cE(x)\coloneqq \bigcap_{i=1}^ {d^d}\{\overline \eta_o^N(p+\ep)\cap \Lambda_N^ {(i)}(x)\ne \emptyset\}. \]
First note that by independence of $\eta_o $ and $\eta_c$, we have
\begin{equation*}
\begin{split}
    \P&[ \cE(x), \Lu_N(x)\cap \overline \eta_c^N(p+\ep)= \emptyset \,|\, \overline \eta ^ N(p)]\\     &\qquad=\P[\cE(x)| \overline \eta ^ N(p)\cap \Lu_N(x)]\cdot  \P[ \Lu_N(x)\cap \overline \eta_c^N(p+\ep)= \emptyset \,|\, \overline \eta ^ N(p)].
    \end{split}
\end{equation*}
Besides, we have
\begin{equation*}
    \P[ \cE(x)| \overline \eta ^ N(p)\cap \Lu_N(x)]\ge \prod_{i=1} ^ {d^ d}\P[\eta_o\cap (\Lambda_N^ {(i)}(x)\times[p,p+\ep])\ne \emptyset]=(1-e^{-\ep (N/d)^ d}) ^ {d^ d}.
\end{equation*}
By Lemma \ref{lem:sprink} together with the two previous inequalities, we obtain
\begin{equation*}
    \P[ \cE(x), \Lu_N(x)\cap \overline \eta_c^N(p+\ep)= \emptyset \,|\, \overline \eta ^ N(p)]\ge (1-e^{-\ep (N/d)^ d}) ^ {d^ d} \delta
\end{equation*}
By independence conditionally on $\overline{\eta}^N(p)$, we have
\begin{equation}\label{ineq:intermediate step2}
    \P\Big[ \bigcap_{\substack{y\in N\Z^d\\ \|y-x\|_\infty=N}} \cE(y)\cap\{\Lu_N(y)\cap \overline \eta_c^N(p+\ep)= \emptyset\} \,\big|\, \overline \eta ^ N(p) \Big] \ge (1-e^{-\ep(N/d)^d})^{3^d d^d} \delta^{3^d} =:\delta'>0.
\end{equation}
It is easy to see that the event on the left hand side of \eqref{ineq:intermediate step2} is contained in the event $\left\{\Lu_N(x)\subset \pazocal{V}_N(p+\ep)\right\}$, thus concluding the proof.
\end{proof}

We are now ready to prove Proposition~\ref{prop:domsto}.

\begin{proof}[Proof of Proposition \ref{prop:domsto}]
Fix $\ep>0$. The stochastic domination \eqref{eq:domsto_1} follows easily from \eqref{eq:sprink_1} together with the observation that conditionally on $\overline \eta ^ N(p)$ the events $\{\Lu_N(x)\subset \overline{\eta}^N_c(p-\ep)\}$, $x\in N\Z^ d$ are all independent.
We now prove \eqref{eq:domsto_2}. 
By Lemma~\ref{lem:sprink_V_N}, there exists $\delta>0$ such that for every $p\in(4^{-d}\ep,1-4^{-d}\ep)$ and $x\in N\Z^d$, we have
\[\P[\Lu_N(Nx)\subset  \pazocal V_N(p+4^ {-d}\ep)| \pazocal V_N(p)]\ge \delta.\]
Notice that for every fixed $r\in\{0,1,2,3\}^d$, conditionally on $\overline{\eta}^N(p)$ all the events 
$\{\Lu_N(Nx)\subset  \pazocal V_N(p+\ep)\}$, $x\in 4\Z^d+r$, are independent. 
We conclude as before that, for every $r\in \{0,1,2,3\}^ d$ and $p\in(4^{-d}\ep,1-4^ {-d}\ep)$, we have
\begin{equation}\label{eq:domsto_partial}
    \pazocal V_N(p+4^ {-d}\ep)\succeq\pazocal V_N(p)\cup (\omega_\delta ^N)_r,
\end{equation}
where
\[(\omega_\delta ^N)_ r\coloneqq \bigcup_{x\in 4\Z^ d+r : (\omega_\delta)_x=1}\Lu_N(Nx).\]
By successively applying \eqref{eq:domsto_partial} for all  $r\in \{0,1,2,3\}^ d$, it yields
\begin{equation}
    \pazocal V_N(p+\ep)\succeq\pazocal V_N(p)\bigcup_{r\in \{0,1,2,3\}^ d} (\omega_\delta ^N)_r=\pazocal V_N(p) \cup \omega_\delta ^N
\end{equation}
for every $p\in(\ep,1-\ep)$.
\end{proof}


\paragraph{Acknowledgements.} We are grateful to an anonymous referee for their valuable comments. This research was supported by the European Research Council (ERC) under the European Union’s Horizon 2020 research and innovation program (grant agreement No 851565).


\end{document}